\documentclass[12pt]{article}
\usepackage{amsthm}
\usepackage{amssymb}
\usepackage{amsfonts}
\usepackage{mathrsfs}
\usepackage[all,cmtip]{xy}
\usepackage{tikz}
\usepackage{tikz-cd}
\usepackage{multirow}
\usepackage{url}
\usepackage{graphicx}
\usepackage{epstopdf}
\usepackage{mathtools}
\usepackage{leftidx}
\usepackage{enumerate}
\usepackage{enumitem}
\usepackage{wrapfig}

\usetikzlibrary{cd}

\usepackage{indentfirst}
\numberwithin{equation}{section}

\DeclareFontFamily{U}{cal}{}
\DeclareFontShape{U}{cal}{m}{n}{<->cmsy10}{}

\DeclareSymbolFont{rcal}{U}{cal}{m}{n}
\DeclareSymbolFontAlphabet{\mathcal}{rcal}

\newtheorem{Def}{Definition}[section]
\newtheorem{Bsp}[Def]{Example}
\newtheorem{Prop}[Def]{Proposition}
\newtheorem{Theo}[Def]{Theorem}
\newtheorem{Lem}[Def]{Lemma}
\newtheorem{Koro}[Def]{Corollary}
\theoremstyle{definition}
\newtheorem{Rem}[Def]{Remark}

\newcommand{\bsm}{\begin{smallmatrix}}
\newcommand{\esm}{\end{smallmatrix}}

\newcommand{\add}{{\rm add}}
\newcommand{\gd}{{\rm gl.dim }}

\newcommand{\End}{{\rm End}}

\def\gldim{\mathop{\rm gl.dim}\nolimits}

\def\ed{\mathop{\rm ext.dim}\nolimits}

\newcommand{\DD}{{\rm D}}

\newcommand{\rad}{{\rm rad}}

\newcommand{\pd}{{\rm pd}}
\newcommand{\id}{{\rm id}}

\newcommand{\DTr}{{\rm DTr}}

\newcommand{\T}{{\mathcal T}}

\newcommand{\cpx}[1]{#1^{\bullet}}
\newcommand{\D}[1]{{\mathscr D}(#1)}

\newcommand{\Db}[1]{{\mathscr D}^b(#1)}
\newcommand{\C}[1]{{\mathscr C}(#1)}

\newcommand{\K}[1]{{\mathscr K}(#1)}

\newcommand{\Kb}[1]{{\mathscr K}^b(#1)}
\newcommand{\modcat}{\ensuremath{\mbox{{\rm -mod}}}}

\newcommand{\stmodcat}[1]{#1\text{{\rm -{\underline{mod}}}}}

\newcommand{\pmodcat}[1]{#1\mbox{{\rm -proj}}}
\newcommand{\imodcat}[1]{#1\mbox{{\rm -inj}}}

\newcommand{\opp}{^{\rm op}}

\newcommand{\Hom}{{\rm Hom}}

\newcommand{\Ext}{{\rm Ext}}

\newcommand{\lra}{\longrightarrow}
\newcommand{\lraf}[1]{\stackrel{#1}{\lra}}
\newcommand{\ra}{\rightarrow}

\title{ \bf  The Extension dimension of syzygy module categories
\\
\footnotetext{Email addresses:  zhengjunling@cjlu.edu.cn, tll878684581@163.com,shuqy@sicnu.edu.cn}
}

\author {Junling  Zheng$^a$, Lulu Tian$^a$, Qianyu Shu$^b$\thanks{Corresponding author} \\
{\it \scriptsize  $^a$ Department of Mathematics, China Jiliang University, Hangzhou, 310018, Zhejiang Province, P. R. China
}\\
{\it \scriptsize  $^b$ School of Mathematical Sciences, Sichuan Normal University, Chengdu, 610000, Sichuan Province, P. R. China
}
}

\date{}
\begin{document}

\maketitle
\begin{abstract}
In this paper, our primary focus is on investigating the extension dimensions of syzygy module categories associated with Artin algebras, particularly under various equivalences. We demonstrate that, for sufficiently large $i$, the $i$-th syzygy module categories of derived equivalent algebras exhibit identical extension dimensions. Furthermore, we establish that the extension dimension of the $i$-th syzygy module category is an invariant under both stable equivalence and separable equivalence for each nonnegative integer $i$.
\end{abstract}

\section{Introduction}

Inspired by the idea of Bondal and van den Bergh in \cite{bb03}, Rouquier introduced the dimension of a triangulated category in \cite{r06,r08}.
This dimension is a measure of the complexity of this category and plays an important role in the representation theory of Artin algebras (see \cite{han2009derived, krause2006rouquier, Ma2024,opp09, r06}). Notably, it can be used to compute the representation dimension of Artin algebras (see \cite{opp09,r06}). Analogous to the dimension of triangulated categories, the extension dimension  of an abelian category was introduced by Beligiannis in \cite{b08}. For an Artin algebra $A$, let $\ed(A)$ denote the extension dimension of the category of all finitely generated left $A$-modules. Beligiannis (\cite{b08}) established that $\ed(A)=0$ if and only if $A$ is representation-finite. This means that the extension dimension of an Artin algebra provides a reasonable way of measuring how far an algebra is from being representation-finite. Recently, many upper bounds have been found for the extension dimension of a given Artin algebra (see \cite{b08,zh22, zheng2020}). However, it would be very hard to give the precise value of dimension of an Artin algebra. The aim of this paper is to provide new information on this problem.

Let $A$ be an Artin algebra and $A\modcat$ be the category of finitely generated left $A$-modules. An $A$-module $K$ is called an $n$\textit{-th syzygy module} ($n>0$) if there is an exact sequence of $A$-modules: $0\ra K\ra P^0\ra P^1\ra \cdots \ra P^{n-1}\ra M\ra 0$ for some $A$-module $M$ with $P^i$ projective; $K$ is called an $\infty$\textit{-th syzygy module} if there is an exact sequence of $A$-modules: $0\ra K\ra P^0\ra P^1\ra P^2 \cdots $ with $P^i$ projective. Denoted by $\Omega^{n}(A\modcat)$
the full subcategory of $A\modcat$
consisting of all $n$-th syzygies $A$-modules and by $\Omega^{\infty}(A\modcat)$
the full subcategory of $A\modcat$
consisting of all $\infty$-th syzygies $A$-modules. For convenience, set $\Omega^{0}(A\modcat):=A\modcat$.
The extended closure properties of the category $\Omega^{n}(A\modcat)$ have been extensively studied, see \cite{ar1994,ar1996,Ber-Ma2023,Shi-Ryo-3016,huang2006}.
Recently, the extension dimension (also known as radius in \cite{dao2014radius})
of the category $\Omega^{n}(A\modcat)$ is also studied (see \cite{dao2014radius,Ma2024,zheng2020}).
In this article, we will continue to study the extension dimensions of syzygy module categories associated with Artin algebras. Specifically, we mainly study the behavior of the extensions dimensions of syzygy module categories under different equivalences.

Derived categories and derived equivalences, introduced by Grothendick and Verdier (\cite{v}), have now connections with various mathematical domains, including algebraic geometry, representation theory of Artin algebras and finite groups (see \cite{h88,huy06,Xi18}), while the Morita theory of derived categories of rings by Rickard (\cite{r1989}) and the Morita theory of derived categories of differential graded algebras by Keller (\cite{k94}) provide a powerful tool to understand homological properties of these equivalent algebras. Notably, the differences of global dimensions, finitistic dimensions and extension dimensions of two derived equivalent algebras are bounded above the length of a tilting complex inducing a derived equivalence (see \cite[Section 12.5(b)]{gr92},\cite{h93,px09,zhang24}). Moreover, numerous homological invariants of derived equivalences have been discovered, including Hochschild homology (\cite{r91}), cyclic homology (\cite{k98}), algebraic $K$-groups (\cite{ds04}) and the number of non-isomorphic simple modules (\cite{r1989}). In this paper, we establish that the extension dimension of the $i$-th syzygy module category is an invariant of derived equivalence for sufficiently large $i$ or when $i=\infty$.

\begin{Theo}\label{main-der-thm}
{\rm (Theorem \ref{der-thm})}
Let $F:\Db{A}\lraf{\sim} \Db{B}$ be a derived equivalence between Artin algebras. Then
$$\ed \Omega^{i}(A\modcat)=\ed\Omega^{i}(B\modcat)$$
for $i$ sufficiently large and
$$\ed \Omega^{\infty}(A\modcat)=\ed\Omega^{\infty}(B\modcat).$$
\end{Theo}

An important problem in representation theory of Artin algebras and groups, is the study of the properties that are invariant under stable equivalence. For example, Mart\'inez-Villa  (\cite{MV1990}) established that stable equivalences preserve the global and dominant dimensions of algebras without nodes. In \cite{Krause00}, Krause and Zwara proved that the representation type is an invariant of stable equivalence. In 2005, Guo (\cite{Guo05}) showed that stable equivalences preserve the representation dimensions of Artin algebras (initially proven by Xi in \cite{Xi02} for stable equivalence of Morita type). In 2022, Xi and Zhang  (\cite{Xi2022}) illustrated that the delooping levels, $\phi$-dimensions and $\psi$-dimensions of Artin algebras are invariants of stable equivalences of algebras without nodes. Recently, Zhang ang Zheng (\cite{zhang24}) verified that the extension dimension of Artin algebras is an invariant of stable equivalence. We will generalize the work of Zhang and Zheng and get the following one of the main results.

\begin{Theo}\label{main-st-thm}
{\rm (Theorem \ref{st-thm} and Theorem \ref{st-thm-node})}
Let $A$ and $B$ be Artin algebras. Suppose that they are stably equivalent. Then $$\ed \Omega^{i}(A\modcat)=\ed \Omega^{i}(B\modcat)$$ for each $i \in \mathbb{N}.$ If, in addition, $A$ and $B$ have no nodes, then $$\ed \Omega^{\infty}(A\modcat)=\ed \Omega^{\infty}(B\modcat).$$
\end{Theo}

Separable equivalences, introduced independently by Kadison (\cite{Kadison91,Kadison95}) and Linckelmann (\cite{Linckelmann11}), provide a powerful tool in the representation theory of Artin algebras and finite groups. For instance, Linckelmann (\cite{Linckelmann11}) demonstrated that certain Hecke algebras possess finitely generated cohomology algebras, and Bergh and Erdmann (\cite{be11}) established a lower bound for the reprsentation dimension of all the classical Hecke algebras of types $A$, $B$ and $D$. Analogous to derived equivalences and stable equivalences, numerous  homological invariants of separable equivalences have been discovered. These include global dimension (\cite{Kadison95}), complexity (\cite{Pea2017}), representation type (\cite{Linckelmann11,Pea2017}) and extension dimension (\cite{zheng2020}). In this paper, we extend the result in \cite{zheng2020}, yielding one of our main results.

\begin{Theo}\label{main-sep-thm}
{\rm (Theorem \ref{sep-thm})}
Let $A$ and $B$ be Artin algebras. Suppose that they are separably equivalent. Then $$\ed \Omega^{i}(A\modcat)=\ed \Omega^{i}(B\modcat)$$ for each $i \in \mathbb{N}\cup \{\infty\}.$
\end{Theo}

The paper is organized as follows: In Section 2, we
recall some basic notations, definitions and facts on extension dimensions. In Section 3, we compare the extension dimensions of syzygy module categories of derived equivalent algebras and prove Theorem \ref{main-der-thm}.
In Section 4, we show Theorem \ref{main-st-thm} and present an example to illustrate this main result.
Finally, the proof of Theorem \ref{main-sep-thm} is establised in Section 5.

\section{Preliminaries}
In this section, we shall fix some notations, and recall some definitions.

\subsection{Stable equivalences and derived equivalences}
Throughout this paper, let $\mathbb{N}$ be the set of natural numbers $\{0,1,2,\cdots\}$, and $R$ be an arbitrary but fixed commutative Artin ring. Unless stated otherwise, all algebras are Artin $R$-algebras with unit, and all modules are finitely generated unitary modules; all categories will be $R$-categories
and all functors are $R$-functors.

Let $A$ be an Artin algebra. We denote by $A\modcat$ the category of all finitely generated left $A$-modules. All subcategories of $A\modcat$ are full, additive and closed under isomorphisms.
For a class of $A$-modules $\mathcal{X}$, we write $\add(\mathcal{X})$ for the smallest full subcategory of $A\modcat$ containing $\mathcal{X}$ and closed under finite direct sums and direct summands.
When $\mathcal{X}$ consists of only one object $X$, we write $\add(X)$ for $\add(\mathcal{X})$.
In particular, $\add({_A}A)$ is exactly the category of projective $A$-modules and also denoted by $A\pmodcat$. We denote by $\mathscr{P}_A$ and $\mathscr{I}_A$ the set of isomorphism classes of indecomposable projective and injective $A$-modules, respectively.
Let $X$ be an $A$-module.
If $f:P\ra X$ is the projective cover of $X$ with $P$ projective, then the kernel of $f$ is called the \emph{syzygy} of $X$, denoted by $\Omega(X)$ (or $\Omega_{A}(X)$).
Dually, if $g:X\ra I$ is the injective envelope of $X$ with $I$ injective, then the cokernel of $g$ is called the \emph{cosyzygy} of $X$, denoted by $\Omega^{-1}(X)$. Additionally, let $\Omega^0$ be the identity functor in $A\modcat$ and $\Omega^1:=\Omega$. Inductively, for any $n\geqslant 2$, define $\Omega^n(X):=\Omega^1(\Omega^{n-1}(X))$ and $\Omega^{-n}(X):=\Omega^{-1}(\Omega^{-n+1}(X))$.
We denoted by $\pd(_AX)$ and $\id(_AX)$ the projective and injective dimension, respectively.

Let $A^{\rm {op}}$ be the opposite algebra of $A$, and $\DD:=\Hom_R(-,E(R/\rad(R)))$ the usual duality from $A\modcat$ to $A^{\rm {op}}\modcat$, where $\rad(R)$ denotes the radical of $R$ and $E(R/\rad(R))$ denotes the injective envelope of $R/\rad(R)$. The duality $\Hom_A(-, A)$ from $\pmodcat{A}$ to $\pmodcat{A\opp}$ is denoted by $^*$, namely for each projective $A$-module $P$, the projective $A\opp$-module $\Hom_A(P, A)$ is written as $P^*$. We write $\nu_A$ for the Nakayama functor $\DD\Hom_A(-,A): \pmodcat{A}\ra \imodcat{A}$.

We denoted by $\stmodcat{A}$ the stable module category of $A$ modulo projective modules. The objects are the same as the objects of $A\modcat$, and the homomorphism set $\underline{\Hom}_A(X,Y)$ between $X$ and $Y$ is given by the quotients of $\Hom_A(X,Y)$  modulo those homomorphisms that factorize through a projective $A$-module. This category is usually called the \emph{stable module category} of $A$. Dually, we denoted by  $A\mbox{-}\overline{\mbox{mod}}$ the stable module category of $A$ modulo injective modules.
Two algebras $A$ and $B$ are said to be \emph{stably equivalent} if the two stable categories $A\stmodcat$ and $B\stmodcat$ are equivalent as additive categories.

\medskip
Let $\cal C$ be an additive category. For two morphisms $f:X\rightarrow Y$ and $g:Y\rightarrow Z$ in $\cal C$, their composition is denoted by $fg$, which is a morphism from $X$ to $Z$. But for two functors $F:\mathcal{C}\ra \mathcal{D}$ and $G:\mathcal{D}\ra\mathcal{E}$ of categories, their composition is written as $GF$.

A complex $\cpx{X}=(X^i, d_X^i)$ over $\mathcal{C}$ is a sequence of objects $X^i$ in
$\cal C$ with morphisms $d_{\cpx{X}}^{i}:X^i\ra X^{i+1}$ such that $d_{\cpx{X}}^{i}d_{\cpx{X}}^{i+1}=0$ for all $i \in {\mathbb Z}$. We denote by $\C{C}$  the category of complexes over $\cal C$ , and by $\K{\mathcal C}$ the homotopy category of complexes over $\mathcal{C}$. If $\cal C$ is an abelian category, then we denote by $\D{\cal C}$ the derived category of complexes over $\cal C$. Let $\Kb{\mathcal C}$
be the full subcategory of $\K{\cal C}$ consisting of
bounded complexes over $\mathcal{C}$.
A complex $X^{\bullet}$ over $\mathcal{C}$ is \emph{cohomologically bounded} if all but finitely many cohomologies of $X^{\bullet}$ are zero. Let $\Db{\cal C}$ be the full subcategory of $\D{\cal C}$ consisting of cohomologically bounded complexes over $\mathcal{C}$. For a given algebra $A$, we simply write $\C{A}$, $\K{A}$ and $\D{A}$ for $\C{A\modcat}$, $\K{A\modcat}$ and $\D{A\modcat}$, respectively. Similarly, we write $\Kb{A}$ and $\Db{A}$ for $\Kb{A\modcat}$ and $\Db{A\modcat}$, respectively.
It is known that $\K{A}$, $\D{A}$, $\Kb{A}$ and $\Db{A}$ are triangulated categories. For a complex $\cpx{X}$ in $\K{A}$ or $\D{A}$, the complex $\cpx{X}[1]$ is obtained from $\cpx{X}$ by shifting $\cpx{X}$ to the left by one degree.

Let $A$ be an Artin algebra.
A homomorphism $f: X\ra Y$ of $A$-modules is said to be a \emph{radical homomorphism} if, for any module $Z$ and homomorphisms $h: Z\ra X$ and $g: Y\ra Z$, the composition $hfg$ is not an isomorphism.
For a complex $(X^i, d_{\cpx{X}}^i)$ over $A\modcat$, if all $d_{\cpx{X}}^i$ are radical homomorphisms, then it is called a \emph{radical complex}, which has the following properties.
\begin{Lem}\label{radical}
{\rm (\cite[pp. 112-113]{hx10})}
Let $A$ be an Artin algebra.

$(1)$ Every complex over $A\modcat$ is isomorphic to a radical complex in $\K{A}$.

$(2)$ Two radical complexes $\cpx{X}$ and $\cpx{Y}$ are isomorphic in $\K{A}$ if and only if they are isomorphic in $\C{A}$.
\end{Lem}

Two algebras $A$ and $B$ are said to be \emph{derived equivalent} if their derived categories $\Db{A}$ and $\Db{B}$ are equivalent as
triangulated categories. In \cite{r1989}, Rickard proved that $A$ and $B$ are derived equivalent if and only if there exists a bounded complex $\cpx{T}$ of finitely generated projective $A$-modules such that $B\cong\End_{\Db{A}}(\cpx{T})$ and

(1) $\Hom_{\Db{A}}(\cpx{T},\cpx{T}[i])=0$ for all $i\ne 0$;

(2) $\Kb{\pmodcat{A}}={\rm thick}(T^{\bullet})$, where ${\rm thick}(T^{\bullet})$ is the smallest triangulated subcategory of $\Kb{\pmodcat{A}}$ containing $T^{\bullet}$ and closed under finite direct sums and direct summands.

{\parindent=0pt} A complex in $\Kb{\pmodcat{A}}$ satisfying the above two conditions is called a \emph{tilting complex} over $A$. It is known that, given a derived equivalence $F:\Db{A}\ra \Db{B}$, there is a unique (up to isomorphism) tilting complex $\cpx{T}$ over $A$ such that $F(\cpx{T})\cong B$ and $F(A)$ is isomorphic in $\Db{B}$ to a tilting complex over $B$.

\begin{Lem}\label{der-lem}{\rm (\cite[Lemma 2.1]{hx10})}
Let $A$ and $B$ be two algebras, and let $F: \Db{A}\ra \Db{B}$ be a derived equivalence with a quasi-inverse $F^{-1}$. Then $F(A)$ is isomorphic in $\Db{B}$ to a complex $\cpx{\bar{T}}\in\Kb{\pmodcat{B}}$  of the form
$$0\lra \bar{T}^0\lra \bar{T}^1\lra\cdots\lra\bar{T}^n\lra 0$$
for some $n\geqslant 0$ if and only if $F^{-1}(B)$ is isomorphic  in $\Db{A}$ to a complex $\cpx{T}\in\Kb{\pmodcat{A}}$ of the form
$$0\lra T^{-n}\lra\cdots\lra T^{-1}\lra T^0\lra 0.$$
\end{Lem}
A special class of derived equivalences can be constructed by tilting modules. Recall that an $A$-module $T$ is said to be a \emph{tilting module} if $T$ satisfies the following three conditions: (1) $\pd(_AT)\leqslant n$, (2) $\Ext_A^i(T,T)=0$ for all $i>0$, and (3) there exists an exact sequence $0\ra A\ra T_0\ra \cdots\ra T_n\ra 0$ in $A\modcat$ with each $T_i$ in $\add(_AT)$. Let $P^{\bullet}(T): 0\ra  P_{n}\ra P_{n-1}\ra \cdots\ra P_{0}\ra 0$ be a projective resolution of $T$. Clearly, $P^{\bullet}(T)$ is a tilting complex over $A$ and $\End_{A}(T) \cong \End_{\Db{A}}(P^{\bullet}(T))$ as algebras.

\subsection{Extension dimensions}
In this subsection, we shall recall the definition and some results of the extension dimensions (see \cite{b08,zheng2020}).

Let $A$ be an Artin algebra. We denote by $A\modcat$ the category of all finitely generated left $A$-modules. For a class $\T$ of $A$-modules, we denote by $\add(\T)$ the smallest full subcategory of $A\modcat$ containing $\T$ and closed under finite direct sums and direct summands. When $\T$ consists of only one object $T$, we write $\add(T)$ for $\add(\T)$.
Let $\T_1,\T_2,\cdots,\T_n$ be subcategories of $A\modcat$. Define
\begin{align*}
\T_1\bullet \T_2
:&=\add(\{X\in A\modcat \mid \mbox{there exists an exact sequence } 0\lra T_1\lra  X \lra T_2\lra 0\\
&\qquad\qquad\qquad\qquad\quad\mbox{in } A\modcat\mbox{ with }T_1 \in \T_1\mbox{ and }T_2 \in \T_2\}).
\end{align*}
The operation $\bullet$ is associative. Inductively, for $n\geqslant 3$, define
$$\T_{1}\bullet  \T_{2}\bullet \dots \bullet\T_{n}:=\add(\{X\in A\modcat \mid \mbox{there exists an exact sequence }
0\lra T_1\lra  X \lra T_2\lra 0$$
$$\mbox{in } A\modcat \mbox{ with }T_1 \in \T_1\mbox{ and }T_2 \in \T_{2}\bullet \dots \bullet\T_{n}\}).$$
Thus $X\in \T_{1}\bullet  \T_{2}\bullet \dots \bullet\T_{n}$ if and only if there exist the following exact sequences
\begin{equation*}
\begin{cases}
\xymatrix@C=1.5em@R=0.1em{
0\ar[r]& T_1 \ar[r]& X\oplus X_1' \ar[r]& X_2 \ar[r]& 0,\\
0\ar[r]& T_2 \ar[r]& X_2\oplus X_2' \ar[r]& X_3 \ar[r]& 0,\\
&&\vdots&&\\
0\ar[r]& T_{n-1} \ar[r]& X_{n-1}\oplus X_{n-1}' \ar[r]& X_n \ar[r]& 0,}
\end{cases}
\end{equation*}
for some $A$-modules $X_i'$ such that $T_i\in \T_i$ and $X_{i+1}\in \T_{i+1}\bullet \T_{i+2}\bullet \cdots \bullet \T_n$ for $1\leqslant i\leqslant n-1$.

For a subcategory $\T$ of $A\modcat$, set $[\T]_{0}:=\{0\}$, $[\T]_{1}:=\add(\T)$,  $[\T]_{n}=[\T]_1\bullet [\T]_{n-1}$ for each $n\geqslant 2$. If $T\in A\modcat$, we write $[T]_{n}$ instead of $[\{T\}]_{n}$.
\begin{Lem}\label{sum-subcat}
Let $A$ be an Artin algebra.
Given two subcategories $\T_{1}$ and $\T_{2}$ of $A\modcat$. For nonnegative integers $m$ and $n$, we have
$$[\T_{1}]_{m}\bullet [\T_{2}]_{n}\subseteq [\T_{1}\oplus\T_{2}]_{m+n},$$
where $\T_{1}\oplus\T_{2}:=\{X\oplus Y\;|\; X\in\T_{1}, Y\in\T_{2}\}.$
\end{Lem}
\begin{proof}
  $[\T_{1}]_{m}\bullet [\T_{2}]_{n}\subseteq [\T_{1}\oplus\T_{2}]_{m}\bullet [\T_{1}\oplus\T_{2}]_{n}\subseteq[\T_{1}\oplus\T_{2}]_{m+n}.$
\end{proof}
Now, let $\T_{1}=\add (T_{1})$ and $\T_{2}=\add (T_{2})$, by Lemma \ref{sum-subcat}, we have
\begin{Koro}{\rm (\cite[Corollary 2.3]{zheng2020})}\label{sum-subseteq}
Let $A$ be an Artin algebra.
Let $T_{1},T_{2}\in A\modcat$. For nonnegative integers $m$ and $n$, we have
$$[T_{1}]_{m}\bullet [T_{2}]_{n}\subseteq [T_{1}\oplus T_{2}]_{m+n}.$$
\end{Koro}
\begin{Lem}{\rm (\cite[Corollary 2.3]{zheng2020})}\label{sum-subseteq-1}
For two module $T_{1}, T_{2}\in A\modcat$ and $m,n\geqslant 1$.
We have $[T_{1}]_{m}\oplus [T_{2}]_{n}\subseteq [T_{1}\oplus T_{2}]_{\max\{m,n\}}.$
\end{Lem}

\begin{Lem}{\rm  (\cite[Lemma 3.5(2)]{zh22})}\label{wrd-lem} Let $A$ be an Artin algebra.
 Let $0\ra M_{n}\ra M_{n-1}\ra \cdots\ra M_0\ra X\ra 0$ be an exact sequence in $A\modcat$. Then
$$X\in [M_0]_1\bullet[\Omega^{-1}(M_{1})]_1
\bullet\cdots\bullet
[\Omega^{-n+1}(M_{n-1})]_1\bullet [\Omega^{-n}(M_n)]_1
\subseteq [\bigoplus_{i=0}^n\Omega^{-i}(M_i)]_{n+1}.$$
\end{Lem}

\begin{Lem}{\rm (\cite[Lemma 3.6]{zh22})}\label{lem-3.6}
Let $X, Y\in A\modcat$  satisfy $[X]_{n_{1}}\subseteq [Y]_{n_{2}}$ with $n_{1}, n_{2}\geqslant 1$.
Then for any $m\geqslant 0$, we have
\begin{itemize}
\item[$(1)$]  $[\Omega^{m}(X)]_{n_{1}}\subseteq[\Omega^{m}(Y)]_{n_{1}n_{2}}$.
\item[$(2)$]  $[\Omega^{-m}(X)]_{n_{1}}\subseteq[\Omega^{-m}(Y)]_{n_{1}n_{2}}$.
\end{itemize}
\end{Lem}

Given a subcategory $\mathcal{C}$, we can define its extension dimension as follows.

\begin{Def}{\rm (\cite[Definition 5.2]{dao2014radius})} \label{def-subcat-extension-dim}
{\rm
Let $A$ be an Artin algebra. Given a subcategory $\mathcal{C}$ of $A\modcat$, we can define its extension dimension
as follows
$$\ed \,\mathcal{C}=\inf\{n\,|\,\mathcal{C}\subseteq [T]_{n+1} \text{ for some nonnegative integer } n \text{ and } T\in A\modcat\}.$$
}
\end{Def}

\begin{Def}\label{ed-def}
{\rm (\cite{b08})}
{\rm The extension dimension of $A\modcat$ is defined to be
\begin{align*}
 \ed(A):= & \ed A\modcat \\
  =& \inf\{n\,|\,A\modcat\subseteq [T]_{n+1} \text{ for some nonnegative integer } n \text{ and } T\in A\modcat\} \\
  =& \inf\{n\geqslant 0\mid A\modcat=[T]_{n+1}\mbox{ with } T\in A\modcat\}.
\end{align*}
}
\end{Def}

\begin{Lem} \label{ext-lem}
Let $A$ be an Artin algebra.

{\rm(1) (\cite[Example 1.6]{b08})} $A$ is representation finite if and only if $\ed(A)=0$.

{\rm(2) (\cite[Example 1.6]{b08})} $\ed(A)\leqslant \ell\ell(A)-1$, where $\ell\ell(A)$ stands for the Loewy length of $A$.

{\rm(3) (\cite[Corollary 3.6]{zheng2020})} $\ed(A)\leqslant \gd(A)$, where $\gd(A)$ stands for the global dimension of $A$.
\end{Lem}

\begin{Lem}\label{subcat-dim}
Let $A$ be an Artin algebra. Given two subcategory $\mathcal{C}, \mathcal{D}$ of $A\modcat$. If
$\mathcal{C}\subseteq \mathcal{D}$, then  we have
$\ed \mathcal{C} \leqslant  \ed \mathcal{D}.$
\end{Lem}
\begin{proof}
Let  $\ed\mathcal{D}=n$, then we have $ \mathcal{D}\subseteq [T]_{n+1}$
for some $T\in A\modcat$ by Definition \ref{def-subcat-extension-dim}.
Then we have $ \mathcal{C}\subseteq [T]_{n+1}$ since $\mathcal{C}\subseteq \mathcal{D}$.
Moreover, we have  $ \ed\mathcal{C}\leqslant n$ by Definition \ref{def-subcat-extension-dim}.
That is, $\ed \mathcal{C} \leqslant  \ed \mathcal{D}.$
\end{proof}

\begin{Prop}\label{subcat-dim2}
Let $A$ be an Artin algebra. Given two subcategory $\mathcal{C}, \mathcal{D}$ of $A\modcat$, we have
$$\max\{\ed \mathcal{C},\ed \mathcal{D}\}\leqslant\ed \mathcal{C}\bullet \mathcal{D} \leqslant \ed \mathcal{C}+\ed \mathcal{D}+1.$$
\end{Prop}
\begin{proof}
Let $\ed\mathcal{C}=m$ and $\ed\mathcal{D}=n$. We can set
$\mathcal{C}\subseteq [T_{1}]_{m+1}$ and $\mathcal{D}\subseteq [T_{2}]_{n+1}$
for some modules $T_{1},T_{2}\in A\modcat$. Then we have
\begin{align*}
\mathcal{C}\bullet \mathcal{D}\subseteq &[T_{1}]_{m+1}\bullet[T_{2}]_{n+1}\\
\subseteq & [T_{1}\oplus T_{2}]_{m+n+2}\ \ \ \ \text{(by Corollary \ref{sum-subseteq})}.
\end{align*}
By Definition \ref{def-subcat-extension-dim}, we have $\ed \mathcal{C}\bullet \mathcal{D} \leqslant m+n+1.$ On the other hand, due to $\mathcal{C}\subseteq \mathcal{C}\bullet \mathcal{D}$ and $\mathcal{D}\subseteq \mathcal{C}\bullet \mathcal{D}$, we have $\ed \mathcal{C}\leqslant\ed \mathcal{C}\bullet \mathcal{D}$ and $\ed \mathcal{D}\leqslant \ed \mathcal{C}\bullet \mathcal{D}$ by Lemma \ref{subcat-dim}. And then $\max\{\ed \mathcal{C},\ed \mathcal{D}\}\leqslant\ed \mathcal{C}\bullet \mathcal{D}$.
\end{proof}
Using mathematical induction and Proposition \ref{subcat-dim2}, we have
\begin{Koro}
Let $A$ be an Artin algebra. Given a positive integer $k$.
For some subcategories $\mathcal{C}_{i}$ of $A\modcat$ for $1\leqslant i \leqslant k$, we have
  $$\max\{\ed\mathcal{C}_{i}\;|\;1\leqslant i \leqslant k\}\leqslant\ed \mathcal{C}_{1}\bullet \mathcal{C}_{2}\bullet\cdots\bullet\mathcal{C}_{k} \leqslant \sum_{i=1}^{k}\ed\mathcal{C}_{i}+k-1.$$
\end{Koro}
\begin{proof}
  Let $\ed\mathcal{C}_{i}=n_{i}$ for each $i.$ We can set
  for some modules $T_{i}\in A\modcat$ for each $i.$  Then we have
   \begin{align*}
  \mathcal{C}_{1}\bullet \mathcal{C}_{2}\bullet\cdots\bullet\mathcal{C}_{k}\subseteq &[T_{1}]_{n_{1}+1}\bullet[T_{2}]_{n_{2}+1}\bullet\cdots \bullet[T_{k}]_{n_{k}+1}\\
\subseteq & [\oplus_{i=1}^{k}T_{i}]_{n_{1}+n_{2}+\cdots+n_{k}+k}\ \ \ \ \text{(by Lemma \ref{sum-subseteq})}.
\end{align*}
By Definition \ref{def-subcat-extension-dim}, we have $\ed \mathcal{C}_{1}\bullet \mathcal{C}_{2}\bullet\cdots\bullet\mathcal{C}_{k}\leqslant n_{1}+n_{2}+\cdots+n_{k}+k-1.$
\end{proof}

\begin{Bsp}{\rm
Let $A$ be the path algebra $kQ$ with quiver $Q$
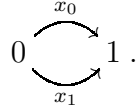
$$\xymatrix{
&0 \ar@/_1pc/[r]_{x_{1}}\ar@/^1pc/[r]^{x_{0}}
&1
}.$$
Note that $A$ is of the infinite representation type and $\gd(A)=1$. By Lemma \ref{ext-lem}(1) and (3), we have $\ed(A)=1$.
Let $\mathcal{C}=\add(S(1))$ and $\mathcal{D}=\add(S(0))$.
For each $M\in A\modcat$, we have
$M\in \mathcal{C}\bullet\mathcal{D}$.
Then $\mathcal{C}\bullet\mathcal{D}=A\modcat$.
On the other hand, $\mathcal{D}\bullet\mathcal{C}=\add(S(0)\oplus S(1))$ since
$S(1)$ is projective and $S(0)$ is injective.
Then $\ed\mathcal{C}=\ed\mathcal{D}=\ed\mathcal{D}\bullet\mathcal{C}=0$.
And we have
 $$\ed \mathcal{C}\bullet \mathcal{D} = \ed \mathcal{C}+\ed \mathcal{D}+1$$
 and
   $$\max\{\ed \mathcal{C},\ed \mathcal{D}\}=\ed \mathcal{D}\bullet \mathcal{C}.$$
}
\end{Bsp}

\begin{Def} {\rm
Let $A$ be an Artin algebra and $A\modcat$ be the category of finitely generated left $A$-modules. An $A$-module $K$ is called an $n$\textit{-th syzygy module} ($n>0$) if there is an exact sequence of $A$-modules: $0\ra K\ra P^0\ra P^1\ra \cdots \ra P^{n-1}\ra M\ra 0$ for some $A$-module $M$ with $P^i$ projective; $K$ is called an $\infty$\textit{-th syzygy module} if there is an exact sequence of $A$-modules: $0\ra K\ra P^0\ra P^1\ra P^2 \cdots $ with $P^i$ projective. Denoted by $\Omega^{n}(A\modcat)$
the full subcategory of $A\modcat$
consisting of all $n$-th syzygies $A$-modules and by $\Omega^{\infty}(A\modcat)$
the full subcategory of $A\modcat$
consisting of all $\infty$-th syzygies $A$-modules.
For convenience, set $\Omega^{0}(A\modcat):=A\modcat$. Clearly, for nonnagative $n$, we have
$$\Omega^{n}(A\modcat)=\{K\oplus P\in A\modcat \mid K\cong\Omega^{n}(M) \text{ for some }A\text{-module } M\text{ and } P\in\pmodcat{A}\}.$$
}\end{Def}
We have the following observation.
\begin{Lem}\label{subseteq-modcat}
Let $A$  be an Artin algebra.
Then for fixed $i\in \mathbb{N}$, we have

$(1)$ $\Omega^{\infty}(A\modcat)\subseteq \Omega^{i+1}(A\modcat)\subseteq \Omega^{i}(A\modcat).$

$(2)$ $\Omega^{i}(\Omega^{n}(A\modcat))=\Omega^{i+n}(A\modcat)$ for each $n\in \mathbb{N}$.

$(3)$  $\Omega^{i}(\Omega^{\infty}(A\modcat))=\Omega^{\infty}(A\modcat)$.
\end{Lem}

\begin{Lem}\label{pro-LL-2}
Let $A$ be a nonsemisimple Artin algebra. Then, for each $i\geqslant 1$,
$$\ed \Omega^{\infty}(A\modcat)\leqslant\ed \Omega^{i}(A\modcat)\leqslant \ell\ell(A)-2.$$
\end{Lem}
\begin{proof}
By Proposition \ref{subcat-dim} and Lemma \ref{subseteq-modcat}(1),
for each $i\geqslant 1$,
$$\ed \Omega^{\infty}(A\modcat)\leqslant\ed \Omega^{i+1}(A\modcat)\leqslant\ed \Omega^{i}(A\modcat).$$
Then we have to show $\Omega(A\modcat)\leqslant \ell\ell(A)-2$.
For each $A$-module $M$, consider the following short exact sequence
$$0\to \Omega(M) \to P \to M \to 0$$
where $P$ is the projective cover of $M$.
Then
$\Omega(M)\subseteq\rad(P)$.
Note that
$$\ell\ell(\Omega(M))\leqslant \ell\ell(\rad(P))=\ell\ell(P)-1\leqslant\ell\ell(A)-1.$$
Defin $n:=\ell\ell(\Lambda)$ and $N:=\Omega(M)$. It follows from $\rad^{n-1}(N)=0$ that there are the following short exact sequences
\begin{equation*}
\begin{cases}
\xymatrix@C=1.5em@R=0.1em{
0\ar[r]& \rad N\ar[r] & N\ar[r]& N/\rad N\ar[r]&0,\\
0\ar[r]& \rad^{2} N\ar[r] & \rad N\ar[r]& \rad N/\rad^{2} N\ar[r]&0,\\
&&\vdots&&\\
0\ar[r]& \rad^{n-2} N\ar[r] & \rad^{n-3}N\ar[r]& \rad^{2} N/\rad^{3} N\ar[r]&0,
}
\end{cases}
\end{equation*}
where $\rad^i(N)/\rad^{i+1}(N)\in [A/\rad(A)]_1$ and $\rad^{n-2}(N)\in [A/\rad(A)]_1$ for $0\le i\le n-3$.
Then $\Omega(M)=N\in [A/\rad(A)]_{n-1}$, and therefore
$\Omega(A\modcat)\subseteq [A/\rad(A)]_{n-1}.$
By Definition \ref{def-subcat-extension-dim}, we have
$\ed \Omega(A\modcat)\leqslant \ell\ell(A)-2$.
\end{proof}

\begin{Prop}\label{syzygy-stablity}
Let $A$ be an Artin algebra. Then there exists a nonnegative integer $m$, such that
$\ed \Omega^{m}(A\modcat)=\ed \Omega^{m+i}(A\modcat)$
for each $i\geqslant 1.$
\end{Prop}
\begin{proof}
By Lemma \ref{ext-lem}(2), we know that $\ed A\modcat<\infty$. It follows from Lemma \ref{subcat-dim} that for $j\in \mathbb{Z}$, $$\ed \Omega^{j+1}(A\modcat)\leqslant \ed \Omega^{j}(A\modcat)<\infty.$$
Thus there exists a nonpositive integer $m$ such that $\ed \Omega^{m}(A\modcat)=\ed \Omega^{m+i}(A\modcat)$
for each $i\geqslant 1.$
\end{proof}

The above proposition tells us that $\ed \Omega^{i}(A\modcat)=\ed \Omega^{i+1}(A\modcat)$ for sufficiently large $i$. A natural question is whether $\ed \Omega^{\infty}(A\modcat)$ is equal to $\ed \Omega^{i}(A\modcat)$ for sufficiently large $i$. The following example gives a negative answer.

\begin{Bsp}{\rm
{\rm (\cite[Example 54]{bm23})}
Let $A =kQ/I$ be an algebra where $Q$ is
$$
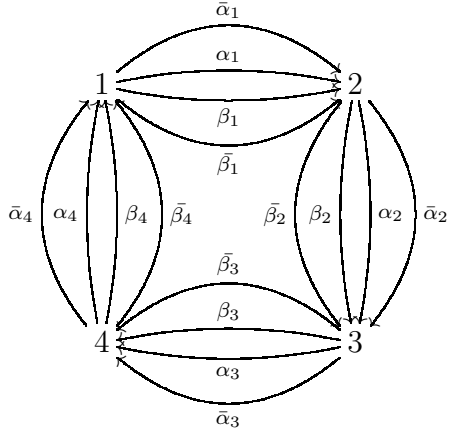
\xymatrix{ 1 \ar@/^8mm/[rrr]^{\bar{\alpha}_1} \ar@/^2mm/[rrr]^{\alpha_1} \ar@/_2mm/[rrr]_{\beta_1} \ar@/_8mm/[rrr]_{\bar{\beta_1}} & &  & 2 \ar@/^8mm/[ddd]^{\bar{\alpha}_2} \ar@/^2mm/[ddd]^{\alpha_2} \ar@/_2mm/[ddd]_{\beta_2} \ar@/_8mm/[ddd]_{\bar{\beta_2}} \\ & &  &\\& & & & \\ 4 \ar@/^8mm/[uuu]^{\bar{\alpha}_4} \ar@/^2mm/[uuu]^{\alpha_4} \ar@/_2mm/[uuu]_{\beta_4} \ar@/_8mm/[uuu]_{\bar{\beta_4}} &  & &  3 \ar@/^8mm/[lll]^{\bar{\alpha}_3} \ar@/^2mm/[lll]^{\alpha_3} \ar@/_2mm/[lll]_{\beta_3} \ar@/_8mm/[lll]_{\bar{\beta_3}}}$$
and $I = \langle \alpha_{i}\alpha_{i+1}-\bar{\alpha}_{i}\bar{\alpha}_{i+1},\ \beta_{i}\beta_{i+1}-\bar{\beta}_{i}\bar{\beta}_{i+1},\ \alpha_{i}\bar{\alpha}_{i+1},\ \bar{\alpha}_{i}\alpha_{i+1},\ \beta_{i}\bar{\beta}_{i+1},\ \bar{\beta}_{i}\beta_{i+1},\text{ for } i \in \mathbb{Z}_4, \ J^3 \rangle$.
It follows from \cite{bm23} that $\Omega^{\infty}(A\modcat)=\pmodcat{A}$ and $\Omega^{n}(A\modcat)$ is infinite representation type for each $n\in\mathbb{N}$. Thus
$$\ed \Omega^{\infty}(A\modcat)=0 \text{ and }\ed  \Omega^{n}(A\modcat)\geqslant 1,\; n\in \mathbb{N}.$$
Note that $A$ is radical cube zero. By Lemma \ref{pro-LL-2}, $\ed \Omega^{n}(A\modcat)\leqslant 1$ for each $n\geqslant 1$. Thus $\ed  \Omega^{n}(A\modcat)=1$ for each $n\geqslant 1$.
}\end{Bsp}

\begin{Lem}\label{lem-syzygy-dimension}
Let $A$ be an Artin algebra. Given a nonnegative integer $m$, we have
$$\ed \Omega^{m-i}(A\modcat)\leqslant \ed \Omega^{m}(A\modcat)+i$$ for each integer $0\leqslant i \leqslant m.$
In particular, $\ed A\modcat\leqslant \ed \Omega^{m}(A\modcat)+m.$
\end{Lem}
\begin{proof}
Suppose $\ed \Omega^{m}(A\modcat)=n$.
Then we have $\Omega^{m}(A\modcat)\subseteq [T]_{n+1}$ for some $T\in A\modcat.$
For each $Y\in \Omega^{m-i}(A\modcat)$, we consider the following exact sequence in $A\modcat$
$$0\to \Omega^{i}(Y)\to P_{i-1}\to \cdots \to P_1\to P_0\to Y\to 0$$
where $P_j\in \pmodcat{A}$ for $0\leqslant j\leqslant i-1$. Note that $\Omega^{i}(Y)\in \Omega^{i}( \Omega^{m-i}(A\modcat))= \Omega^{m}(A\modcat)\subseteq [T]_{n+1}$.
Then
\begin{align*}
Y\in &[P_{0}]_{1}\bullet [\Omega^{-1}(P_{1})]_{1}\bullet\cdots\bullet  [\Omega^{-(i-1)}(P_{i-1})]_{1}\bullet [ \Omega^{-i}(\Omega^{i}(Y))]_{1}\ \ \ \ \text{(by Lemma \ref{wrd-lem})}\\
\subseteq & [P_{0}]_{1}\bullet [\Omega^{-1}(P_{1})]_{1}\bullet\cdots\bullet  [\Omega^{-(i-1)}(P_{i-1})]_{1}\bullet [ \Omega^{-i}(T)]_{n+1}\ \ \ \ \text{(by Lemma \ref{lem-3.6})}\\
\subseteq & [A]_{1}\bullet [\Omega^{-1}(A)]_{1}\bullet\cdots\bullet  [\Omega^{-(i-1)}(A)]_{1}\bullet [ \Omega^{-i}(T)]_{n+1}\\
\subseteq &[A\oplus \Omega^{-1}(A)\oplus\cdots\oplus\Omega^{-(i-1)}(A)\oplus \Omega^{-i}(T)]_{n+i+1}\ \ \ \ \text{(by Corollary \ref{sum-subseteq})}\\
= &[\big(\bigoplus_{j=0}^{i-1}\Omega^{-j}(A)\big)\oplus \Omega^{-i}(T)]_{n+i+1}.
\end{align*}
Then $\Omega^{m-i}(A\modcat)\subseteq [\big(\bigoplus_{j=0}^{i-1}\Omega^{-j}(A)\big)\oplus \Omega^{-i}(T)]_{n+i+1}.$
Thus $\ed \Omega^{m-i}(A\modcat)\leqslant n+i$.
\end{proof}

\begin{Def}\label{def-syzygy-modcat}
{\rm
A subcategory $\mathcal{C}$ of $A\modcat$ is \emph{n-syzygy-finite} if there is some nonnegative integer $n$ such that $\Omega^n(\mathcal{C})$ is representation-finite, that is, the number of non-isomorphic indecomposable direct summands of objects in $\Omega^n(\mathcal{C})$ is finite. An algebra $A$ is \emph{syzygy-finite} if $A\modcat$ is $n$-syzygy-finite
 for some nonnegative integer $n$.}
\end{Def}
\begin{Theo}
Let $A$ be an Artin algebra. Suppose that $A$ is $n$-syzygy finite and  $\ed(A)=n$.
Then $\ed \Omega^{i}(A\modcat)=n-i$ for each $0\leqslant i \leqslant n.$
\end{Theo}
\begin{proof}
As $A$ is $n$-syzygy finite, $\ed \Omega^{n}(A\modcat)=0$.
By Lemma \ref{lem-syzygy-dimension},
$\ed \Omega^{i}(A\modcat)\leqslant n-i$. Then $\ed \Omega^{i}(A\modcat)=n-i$. Otherwise, suppose $\ed \Omega^{i}(A\modcat)< n-i$. By Lemma \ref{lem-syzygy-dimension},
$$\ed(A)=\ed A\modcat \leqslant \ed \Omega^{i}(A\modcat)+i<(n-i)+i=n,$$
a contradiction to the assumption $\ed(A)=n$.
\end{proof}
\begin{Koro}\label{extension-global}
Let $A$ be an Artin algebra. Suppose that the global dimension $\gldim (A)$ of $A$ is finite and  $\ed(A)=\gd(A)$.
Then $\ed \Omega^{i}(A\modcat)=\gldim (A)-i$
for each $0\leqslant i \leqslant \gldim (A).$
\end{Koro}

\begin{Bsp}{\rm
Let $A$ be the Beilinson algebra $kQ/I$ with quiver $Q$
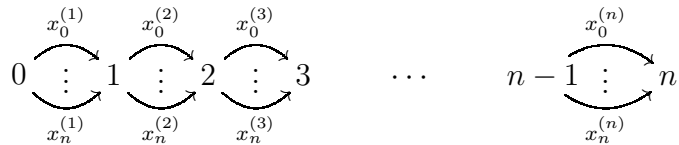
$$\xymatrix{
&0 \ar@/_1pc/[r]_{x_{n}^{(1)}}\ar@/^1pc/[r]^{x_{0}^{(1)}}_{\vdots}
&1\ar@/_1pc/[r]_{x_{n}^{(2)}}\ar@/^1pc/[r]^{x_{0}^{(2)}}_{\vdots}
&2\ar@/_1pc/[r]_{x_{n}^{(3)}}\ar@/^1pc/[r]^{x_{0}^{(3)}}_{\vdots}
&3&\cdots &n-1\ar@/_1pc/[r]_{x_{n}^{(n)}}\ar@/^1pc/[r]^{x_{0}^{(n)}}_{\vdots}&n
}$$
and relations $I=(x_{i}^{(l)}x_{j}^{(l+1)}-x_{j}^{(l)}x_{i}^{(l+1)})$ for $i,j\in\{0,1,2,\cdots,n\},l\in\{1,2,3,\cdots,n-1\}$. By \cite[Example 3.4]{zh22}, we know that $\ed(A)=n=\text{\rm gldim }(A)$.
Then, by Corollary \ref{extension-global}, we have $\ed \;\Omega^{i}(A\modcat)=n-i$ for each $0\leqslant i \leqslant n.$
}
\end{Bsp}

\section{Derived equivalences}
In this section, we discuss the relationships of the extension dimensions of the $i$-th syzygy module categories associated with two derived equivalent algebras. We first introduce some notationas and basic facts related to derived euivalences, as detailed in reference \cite{hp17,hx10}.

\begin{Def}{\rm A derived equivalence $F:\Db{A}\ra \Db{B}$ is called \emph{nonnegative} if

$(1)$ $F(X)$ is isomorphic to a complex with zero homology in all negative degrees for all $X\in A\modcat$; and
$(2)$ $F(P)$ is isomorphic to a complex in $\Kb{\pmodcat{A}}$ with zero terms in all negative degrees for all $P\in\pmodcat{A}$.
}
\end{Def}
\begin{Lem}\label{lem-nonneg}
{\rm (\cite[Lemma 4.2]{hp17})} A derived equivalence $F:\Db{A}\ra \Db{B}$ is nonnegative if and only if $F(A)$ is isomorphic in $\Kb{\pmodcat{B}}$ to a complex with zero terms in all positive degrees.
In particular, $F[i]$ is nonnegative for sufficiently small $i$.
\end{Lem}
For every nonnegative derived equivalence $F$, Hu-Xi (see \cite[Section 3]{hx10}) construct a functor  $\overline{F}:\stmodcat{A}\ra \stmodcat{B}$, which is called the \emph{stable functor} of $F$.
This stable functor has the following properties.

\begin{Lem}\label{lem-non}{\rm (see \cite{hx10} or \cite[Section 4]{hp17})}
$(1)$ Let $i$ be a nonnegative integer. Then $i$-th syzygy functor $\Omega^{i}_{A}:\stmodcat{A}\ra \stmodcat{A}$ is a stable functor of the derived equivalence $[-i]:\Db{A}\ra \Db{A}$, that is, $\overline{[-i]}\cong \Omega^i_A$ as additive functors. Particularly, the stable functor of identity functor on $\Db{A}$ is isomorphic to the identity functor on $\stmodcat{A}$.

$(2)$ Let $F:\Db{A}\ra \Db{B}$ and $G:\Db{B}\ra \Db{C}$ be two nonnegative derived equivalences. Then the functors $\overline{G}\circ \overline{F}$ and $\overline{GF}$ are isomorphic.

$(3)$ Let $F:\Db{A}\ra \Db{B}$ be a nonnegative derived equivalence. Suppose that $$0 \lra X\lra Y \lra Z\lra 0$$ is an exact sequence in $A\modcat$. Then there is an exact sequence
$$0 \lra \overline{F}(X) \lra \overline{F}(Y)\oplus Q \lra \overline{F}(Z)\lra 0$$ in $B\modcat$ for some projective $B$-module $Q$.

$(4)$ Given two Artin algebras $A$ and $B$.
Let $F:D^{b}(A\modcat)\ra D^{b}(B\modcat)$
be a triangle functor.
Then $\overline{F}\circ \Omega_{A}\cong  \Omega_{B}\circ \overline{F}$. In general, $\overline{F}\circ \Omega^i_{A}\cong  \Omega^i_{B}\circ \overline{F}$ for $i\in\mathbb{N}$.
\end{Lem}

\begin{Theo}\label{der-thm} Let $F:\Db{A}\lraf{\sim} \Db{B}$ be a derived equivalence between Artin algebras. Then

$(1)$ $\ed \Omega^{i}(A\modcat)=\ed\Omega^{i}(B\modcat)$ with $i$ large enough.

$(2)$ $\ed \Omega^{\infty}(A\modcat)=\ed\Omega^{\infty}(B\modcat).$
\end{Theo}
\begin{proof}
Set $p:=\ell(F(A))-1$. Let $\cpx{T}$ and $\cpx{\bar{T}}$ be the radical tilting complexes associated to $F$ and the quasi-inverse $G$ of $F$, respectively. By applying the shift functor, we can assume that $\cpx{T}\in\Kb{\pmodcat{A}}$ is of the form
$$0\lra T^{-p}\lra\cdots\lra T^{-1}\lra T^0\lra 0.$$
By Lemma \ref{der-lem}, $\cpx{\bar{T}}\in\Kb{\pmodcat{B}}$ is of the form
$$0\lra \bar{T}^0\lra \bar{T}^1\lra\cdots\lra\bar{T}^p\lra 0.$$
By Lemma \ref{lem-nonneg}, we know that $F$ and $G[-p]$ are nonnegative.

$(1)$ By Proposition \ref{syzygy-stablity}, we can take two positive integer $m_{1}$ and $m_{2}$, such that
$$\ed \Omega^{m_{1}}(A\modcat)=\ed \Omega^{m_{1}+i}(A\modcat)=n_{1}$$ and $$\ed \Omega^{m_{2}}(B\modcat)=\ed \Omega^{m_{2}+i}(B\modcat)=n_{2}$$
for each $i\geqslant 1.$

We set $\Omega_{B}^{m_{2}}(B\modcat)\subseteq [T]_{n_{2}+1}$. For $\Omega_{A}^{m_{1}}(X)\in\Omega_{A}^{m_{1}}(A\modcat) $ with
$X\in A\modcat $, we have $\overline{F}(\Omega^{m_{1}}_{A}(X))\in B\modcat $, and we have the following
 short exact sequences in $B\modcat $
\begin{equation*}
\begin{cases}
\xymatrix@C=1.5em@R=0.1em{
0\ar[r]& X_{1}\ar[r] & \Omega^{m_{2}}_{B}(\overline{F}(\Omega^{m_{1}}_{A}(X)))\oplus X'\ar[r]& Y_{1}\ar[r]&0,{\rm \;\;where \;\;} X_{1}\in[T]_{1}{\rm \;\;and\;\; } Y_{1}\in[T]_{n_{2}},\\
0\ar[r]& X_{2}\ar[r] & Y_{1}\oplus Y_{1}'\ar[r]& Y_{2}\ar[r]&0,{\rm \;\;where \;\;} X_{2}\in[T]_{1}{\rm \;\;and\;\; } Y_{2}\in[T]_{n_{2}-1},\\
&&\vdots&&\\
0\ar[r]& X_{n_{2}}\ar[r] & Y_{n_{2}-1}\oplus Y_{n_{2}-1}'\ar[r]& Y_{n_{2}}\ar[r]&0,{\rm \;\;where \;\;} X_{n_{2}}\in[T]_{1}{\rm \;\;and\;\; } Y_{n_{2}}\in[T]_{1}.
}
\end{cases}
\end{equation*}
Applying the stable functor $\overline{G[-p]}$ of $G[-p]$ to the above short exact sequences and by Lemma \ref{lem-non}(3), we can get the following short exact sequences
\begin{equation*}
\begin{cases}
\xymatrix@C=1.5em@R=0.1em{
0\ar[r]& \overline{G[-p]}(X_{1})\ar[r] & \overline{G[-p]}(\Omega^{m_{2}}_{B}(\overline{F}(\Omega^{m_{1}}_{A}(X))))\oplus \overline{G[-p]}(X')\oplus P_{1}\ar[r]& \overline{G[-p]}(Y_{1})\ar[r]&0,\\
0\ar[r]& \overline{G[-p]}(X_{2})\ar[r] & \overline{G[-p]}(Y_{1})\oplus \overline{G[-p]}(Y_{1}')\oplus P_{2}\ar[r]& \overline{G[-p]}(Y_{2})\ar[r]&0,\\
&&\vdots&&\\
0\ar[r]& \overline{G[-p]}(X_{n_{2}})\ar[r] & Y_{n_{2}-1}\oplus \overline{G[-p]}(Y_{n_{2}-1}')\oplus P_{n_{2}}\ar[r]& \overline{G[-p]}(Y_{n_{2}})\ar[r]&0,
}
\end{cases}
\end{equation*}
where $\overline{G[-p]}(X_{i})\in [\overline{G[-p]}(T)]_{1}$ and $\overline{G[-p]}(Y_{i})\in [\overline{G[-p]}(T)]_{n_{2}+1-i}$
and $P_{i}\in \pmodcat{B}$
for each $1\leqslant i \leqslant n_{2}.$
In particular, $$\overline{G[-p]}\Omega_{B}^{m_{2}}(\overline{F}(\Omega_{A}^{m_{1}}(X)))\in [\overline{G[-p](T)}]_{n_{2}+1}.$$
On the other hand, we have the following isomorphisms in $\stmodcat{A}$
\begin{align*}
	\overline{G[-p]}\Omega_{B}^{m_{2}}(\overline{F}(\Omega_{A}^{m_{1}}(X)))
	&\cong(\overline{G[-p]}\circ \overline{[-m_{2}]}\circ \overline{F})(\Omega_{A}^{m_{1}}(X))\quad \mbox{(by Lemma \ref{lem-non}(1))}\\
	&\cong \overline{G[-p]\circ [-m_{2}]} \circ \overline{F}(\Omega_{A}^{m_{1}}(X))\quad \mbox{(by Lemma \ref{lem-non}(2))}\\
	&\cong \overline{G[-p-m_{2}]} \circ \overline{F}(\Omega_{A}^{m_{1}}(X))\\
	&\cong \overline{G[-p-m_{2}]\circ F}(\Omega_{A}^{m_{1}}(X))\quad \mbox{(by Lemma \ref{lem-non}(2))}\\
	&\cong \overline{G\circ F\circ[-p-m_{2}]}(\Omega_{A}^{m_{1}}(X))\\
	&\cong \overline{{\rm Id}_{\Db{A}}[-p-m_{2}]}(\Omega_{A}^{m_{1}}(X))
	\quad \mbox{(by Lemma \ref{lem-non}(1))}\\
	&\cong \overline{[-p-m_{2}]}(\Omega_{A}^{m_{1}}(X))
	\\
	&\cong \Omega_{A}^{p+m_{2}}(\Omega_{A}^{m_{1}}(X))		\quad \mbox{(by Lemma \ref{lem-non}(1))}
	\\
	&\cong \Omega_{A}^{m_{1}+p+m_{2}}(X). 
\end{align*}
By \cite[Theorem 2.2]{Heller60}, there are projective $A$-modules $Q$ and $Q'$ such that $$\overline{G[-p]}\Omega_{B}^{m_{2}}(\overline{F}(\Omega_{A}^{m_{1}}(X)))\oplus Q\cong \Omega_{A}^{m_{1}+p+m_{2}}(X)\oplus Q'$$ as $A$-modules.
Moreover, we  can get
$$\Omega_{A}^{m_{1}+p+m_{2}}(X)\in [\overline{G[-p](T)}\oplus A]_{n_{2}+1}.$$
Note that $$\Omega^{m_{1}+p+m_{2}}(A\modcat)=\{K\oplus P \in A\modcat \mid K\cong\Omega^{m_{1}+p+m_{2}}(M) \text{ for some }A\text{-module } M \text{ and } P\in \pmodcat{A}\}.$$ Then
$$\Omega_{A}^{m_{1}+p+m_{2}}(A\modcat)\subseteq [\overline{G[-p](T)}\oplus A]_{n_{2}+1}.$$
And by Definition \ref{def-subcat-extension-dim}, we have
$$\ed\Omega_{A}^{m_{1}}(A\modcat)=\ed \Omega_{A}^{m_{1}+p+m_{2}}(A\modcat)\leqslant n_{2}=\ed\Omega_{A}^{m_{2}}(B\modcat ).$$
That is, $\ed\Omega_{A}^{m_{1}}(A\modcat )\leqslant \ed\Omega_{B}^{m_{2}}(B\modcat).$
Similarly, we also have
$$\ed\Omega_{B}^{m_{2}}(B\modcat )\leqslant\ed\Omega_{A}^{m_{1}}(A\modcat).$$
And then we get $\ed\Omega_{A}^{m_{1}}(A\modcat)= \ed\Omega_{B}^{m_{2}}(B\modcat).$ By Proposition \ref{syzygy-stablity}, for $i$ sufficiently large, we have $$\ed \Omega^{i}(A\modcat)=\ed\Omega^{i}(B\modcat).$$

$(2)$ Let $X\in \Omega^{\infty}(A\modcat)$. We claim that
$\overline{F}(X)\in \Omega^{\infty}(B\modcat)$. Indeed, it follows from $X\in \Omega^{\infty}(A\modcat)$ that the
following exact sequence
$$0\lra X\lra P^{0}\lraf{f^0} P^{1}\lraf{f^1} P^{2}\lraf{f^2}\cdots$$
with $P^{i}\in\pmodcat{A}$ for each $i\in \mathbb{N}$.
Let $K^i$ be the kernal of $f^i$. Then $K^0=X$ and we have the following short exact sequences
\begin{equation*}
\begin{cases}
\xymatrix@C=1.5em@R=0.1em{
0\ar[r]& X\ar[r] & P^{0}\ar[r]& K^{1}\ar[r]&0,\\
0\ar[r]& K^{1}\ar[r] & P^{1}\ar[r]& K^{2}\ar[r]&0,\\
0\ar[r]& K^{2}\ar[r] & P^{2}\ar[r]& K^{3}\ar[r]&0,\\
&  &\vdots & &.
}
\end{cases}
\end{equation*}
Applying the functor $\overline{F}$ to the above short exact sequences and by Lemma \ref{lem-non}(3), we can get the following short exact sequences
\begin{equation*}
\begin{cases}
\xymatrix@C=1.5em@R=0.1em{
0\ar[r]& \overline{F}(X)\ar[r]     & Q^{0}\ar[r]& \overline{F}(K^{1})\ar[r]&0,\\
0\ar[r]& \overline{F}(K^{1})\ar[r] & Q^{1}\ar[r]& \overline{F}(K^{2})\ar[r]&0,\\
0\ar[r]& \overline{F}(K^{2})\ar[r] & Q^{2}\ar[r]& \overline{F}(K^{3})\ar[r]&0,\\
&  &\vdots & &,
}
\end{cases}
\end{equation*}
where $_BQ^{i}$ is projective for each $i\in\mathbb{N}$.
And we get the following exact sequence
$$\xymatrix@C=1.5em@R=0.1em{
0\ar[r]&\overline{F}(X)\ar[r]
&Q^{0}\ar[r]&
Q^{1}\ar[r]&Q^{2}\ar[r]& \cdots.
}$$
Then  $\overline{F}(X)\in \Omega^{\infty}(B\modcat)$.

Now, let $\ed \Omega^{\infty}(B\modcat)=r_{2}$ and $\Omega^{\infty}(B\modcat)\subseteq [L]_{r_{2}+1}$
for some $L\in B\modcat$. It follows from $\overline{F}(X)\in \Omega^{\infty}(B\modcat)\subseteq [L]_{r_{2}+1}$ that there are the following exact sequences in $B\modcat$.
\begin{equation*}
\begin{cases}
\xymatrix@C=1.5em@R=0.1em{
0\ar[r]& L_{1}\ar[r] & \overline{F}(X)\oplus Z'_0\ar[r]& Z_{1}\ar[r]&0,\\
0\ar[r]& L_{2}\ar[r] & Z_{1}\oplus Z_{1}'\ar[r]& Z_{2}\ar[r]&0,\\
&&\vdots&&\\
0\ar[r]& L_{r_{2}}\ar[r] & Z_{r_{2}-1}\oplus Z_{r_{2}-1}'\ar[r]& Z_{r_{2}}\ar[r]&0.
}
\end{cases}
\end{equation*}
Here $L_i\in [L]_1$ and $Z_i\in [L]_{r_2-i+1}$ for $1\leqslant i\leqslant r_2$.
Applying the stable functor $\overline{G[-p]}$ of $G[-p]$ to the above short exact sequences and by Lemma \ref{lem-non}(3), we can get the following short exact sequences
\begin{equation*}
\begin{cases}
\xymatrix@C=1.5em@R=0.1em{
0\ar[r]& \overline{G[-p]}(L_{1})\ar[r] & \overline{G[-p]}(\overline{F}(X))\oplus \overline{G[-p]}(Z'_0)\oplus P'_1\ar[r]& \overline{G[-p]}(Z_{1})\ar[r]&0,\\
0\ar[r]& \overline{G[-p]}(L_{2})\ar[r] & \overline{G[-p]}(Z_{1})\oplus \overline{G[-p]}(Z_{1}')\oplus P'_2\ar[r]& \overline{G[-p]}(Z_{2})\ar[r]&0,\\
&&\vdots&&\\
0\ar[r]& \overline{G[-p]}(L_{r_{2}})\ar[r] & \overline{G[-p]}(Z_{r_{2}-1})\oplus \overline{G[-p]}(Z_{r_{2}-1}')\oplus P'_{r_2}\ar[r]& \overline{G[-p]}(Z_{r_{2}})\ar[r]&0,
}
\end{cases}
\end{equation*}
where $P'_{i}\in \pmodcat{A}$ and $\overline{G[-p]}(L_i)\in [\overline{G[-p]}(L)]_1$ and $\overline{G[-p]}(Z_i)\in [\overline{G[-p]}(L)]_{r_2-i+1}$ for $1\leqslant i\leqslant r_2$.
Thus $$\overline{G[-p]}(\overline{F}(X))\in [\overline{G[-p]}(L)]_{r_2+1}.$$
On the other hand, we have the following isomorphisms in $\stmodcat{A}$.
\begin{align*}
\overline{G[-p]}(\overline{F}(X))
&\cong \overline{G}(\overline{[-p]}\overline{F})(X)
\quad \mbox{(by Lemma \ref{lem-non}(2))}\\
&\cong \overline{G}(\Omega_{B}^{p}\overline{F})(X)
\quad \mbox{(by Lemma \ref{lem-non}(1))}\\
&\cong \overline{G}(\overline{F}\Omega_{A}^{p})(X)
\quad \mbox{(by Lemma \ref{lem-non}(4))}\\
&\cong \overline{GF}\Omega_{A}^{p}(X)
\quad \mbox{(by Lemma \ref{lem-non}(2))}\\
&\cong \overline{{\rm Id}_{\Db{A}}}\Omega_{A}^{p}(X)\\
&\cong \Omega_{A}^{p}(X)
\quad \mbox{(by Lemma \ref{lem-non}(1))}
\end{align*}
Thus $\Omega_{A}^{p}(X)\in [\overline{G[-p]}(L)\oplus A]_{r_2+1}$ and $\Omega^{p}_{A}(\Omega^{\infty}(A\modcat))\subseteq [\overline{G[-p]}(L)\oplus A]_{r_2+1}$.
By Lemma \ref{subseteq-modcat}(3),  $$\Omega^{p}_{A}(\Omega^{\infty}(A\modcat))=\Omega^{\infty}(A\modcat)$$
Then $\Omega^{\infty}(A\modcat)\subseteq [\overline{G[-p]}(L)\oplus A]_{r_2+1}.$
Thus $$\ed \Omega^{\infty}(A\modcat)\leqslant r_{2}=\ed \Omega^{\infty}(B\modcat).$$
Similarly, we also have
$\ed \Omega^{\infty}(B\modcat)\leqslant \ed \Omega^{\infty}(A\modcat)$.
And then $$\ed \Omega^{\infty}(A\modcat)=\ed \Omega^{\infty}(B\modcat).$$
This finishes the proof.
\end{proof}

As an immediate consequence of Theorem \ref{der-thm}, we have
\begin{Koro}
Let $A$ be an Artin algebra, $T$ be a tilting $A$-module and $B=\End_{A}(T)$. Then we have

$(1)$ $\ed\Omega_{A}^{m_{1}}(A\modcat)= \ed\Omega_{A}^{m_{2}}(B\modcat)$ for some integer $m_{1},m_{2}\in\mathbb{N}.$

$(2)$ $\ed\Omega_{A}^{\infty}(A\modcat)= \ed\Omega_{A}^{\infty}(B\modcat)$.
\end{Koro}
\begin{proof}
Let $p:=\pd(_AT).$ We have the following  minimal projective resolution of $_AT$
$$0\lra  Q_{p}\lra Q_{n-1}\lra \cdots\lra Q_{0}\lra T\lra 0.$$ Moreover, the following complex
$$Q^{\bullet}(T):  0\lra  Q_{p}\lra Q_{n-1}\lra \cdots\lra Q_{0}\lra 0 $$
is a tilting complex in $\Db{A}$ and $B\cong \End_{\Db{A}}(Q^{\bullet}(T))$. Thus, by Theorem \ref{der-thm}, we get the result.
\end{proof}

\begin{Koro}\label{cor-syzygy}
{\rm (\cite{hp17,Wei2018Derived})}
If $A$ and $B$ are derived-equivalent, then $A$ is syzygy-finite if and only if so $B$ does.
\end{Koro}
\begin{proof}
We have to show that if $A$ is syzygy-finite, then $B$ is syzygy-finite. Indeed, since  $A$ is syzygy-finite, there is some integer $n$ such that $\Omega^n(A\modcat)$ is representation-finite. In particular, $\ed \Omega^n(A\modcat)=0$. By Theorem \ref{der-thm}, $\ed \Omega^m(B\modcat)=0$ for some $m$. Thus $\Omega^m(B\modcat)$ is representation-finite and $B$ is syzygy-finite.
\end{proof}

To illustrate Theorem \ref{der-thm}, we give the following example.

\begin{Bsp}{\rm (\cite[Example 3.11]{zhang24})
Let $A$ be an algebra over a field $k$ given by the following quiver $Q_A$
$$\xymatrix{
&&3\ar[dl]_{\beta_1}\\
1&2\ar[l]|-{\alpha}&4\ar[l]|-{\beta_2}\\
&&5\ar[ul]^{\beta_3}
}$$
with relations $\alpha\beta_i=0,\;1\leqslant i\leqslant 3$. The Aulander-Reiten quiver of $A\modcat$ is as follows:
$$\xymatrix@C=1.5em@R=1.5em{
&
{\text{$\begin{smallmatrix}
2\\
1
\end{smallmatrix}$}}\ar[ddr]
&&&&&&\\
&&&
{\text{$\begin{smallmatrix}
3\\
2
\end{smallmatrix}$}}\ar[dr]
&&
{\text{$\begin{smallmatrix}
4&&5\\
&2&
\end{smallmatrix}$}}\ar[dr]\ar@{-->}[ll]
&&
{\text{$\begin{smallmatrix}
3
\end{smallmatrix}$}}\ar@{-->}[ll]
\\
{\text{$\begin{smallmatrix}
1
\end{smallmatrix}$}}\ar[uur]
&&
{\text{$\begin{smallmatrix}
2
\end{smallmatrix}$}}\ar[dr]\ar[ddr]\ar[ur]\ar@{-->}[ll]
&&
{\text{
$\begin{smallmatrix}
3&&4&&5 \\
&2&&2&
\end{smallmatrix}$}}\ar[dr]\ar[ddr]\ar[ur]\ar@{-->}[ll]
&&
{\text{
$\begin{smallmatrix}
3&4&5 \\
&2&
\end{smallmatrix}$}}\ar[dr]\ar[ddr]\ar[ur]\ar@{-->}[ll]
&
\\
&&&
{\text{
$\begin{smallmatrix}
4 \\
2
\end{smallmatrix}$}}\ar[ur]
&&
{\text{
$\begin{smallmatrix}
3&&5 \\
&2&
\end{smallmatrix}$}}\ar[ur]\ar@{-->}[ll]
&&
{\text{
$\begin{smallmatrix}
4
\end{smallmatrix}$}}\ar@{-->}[ll]
\\
&&&
{\text{
$\begin{smallmatrix}
5\\
2
\end{smallmatrix}$}}\ar[uur]
&&
{\text{
$\begin{smallmatrix}
3&&4 \\
&2&
\end{smallmatrix}$}}\ar[uur]\ar@{-->}[ll]
&&
{\text{
$\begin{smallmatrix}
5
\end{smallmatrix}$}}\ar@{-->}[ll]
}$$
Thus $A$ is representation-finite and $\ed( A)=0$ by Lemma \ref{ext-lem}(1).  Then $\ed \Omega^{i}(A\modcat)=0$ for each $i\geqslant 0$.
Let
$T:=
2\oplus \begin{matrix}2\\1\end{matrix}
\oplus \begin{matrix}3\\2\end{matrix}
\oplus \begin{matrix}4\\2\end{matrix}
\oplus \begin{matrix}5\\2\end{matrix}$.
By calculation, we obtain that $T$ is a tilting module with $\pd(_AT)=1$.
Let $P^{\bullet}$ be the projective resolution of $T$. Then $P^{\bullet}$ is a $2$-term tilting complex and $B:=\End_{\Db{A}}(P^{\bullet})$ is the path algebra given by the quiver $Q_B$
$$\xymatrix{
&&c\ar[dl]\\
b&a\ar[l]&d\ar[l]\\
&&e\ar[ul]
}$$
Since the underlying graph of $Q_B$ is Euclidean, we know that $B$ is a hereditary $k$-algebra of infinite representation type
and $\ed(B\modcat)=1$ by Lemmas \ref{ext-lem}(1) and \ref{ext-lem}(3). Due to $\gd(B)=1$, $\Omega(B\modcat)\subseteq\pmodcat{B}$ and $\ed \Omega(B\modcat)=0$. In particular, $\ed \Omega^i(B\modcat)=0$ for each $i\geqslant 1$.
Thus for each $m_1\geqslant 0$ and $m_2\geqslant 1$, $\ed \Omega^{m_{1}}(A\modcat)=\ed\Omega^{m_{2}}(B\modcat).$
}
\end{Bsp}

\begin{Bsp}{\rm(\cite[Example 4.25]{Xi18})
Let $A$ be a $k$-algebra given by the quiver with relations:
$$\begin{array}{ccc}
\xymatrix{
1\ar@<2.5pt>[r]^{\gamma}
&2\ar@<2.5pt>[l]^{\beta}
\ar@<2.5pt>[r]^{\delta} &3\ar@<2.5pt>[l]^{\alpha}\ar@(ru,rd)^{\epsilon}\\
&&
4\ar[u]_{\eta}
}
\\
\alpha\delta\alpha, \delta\gamma, \alpha\delta-\gamma\beta,\epsilon^n,\epsilon\delta,\alpha\epsilon,\alpha\eta.
\end{array}$$
We will show that $A$ is syzygy-finite. It follows from \cite[Example 4.25]{Xi18} that $A$ is derived equivalent to a $k$-algebra $B$ given by the quiver with relations:
$$\begin{array}{ccc}
\xymatrix{
2'\ar[r]^{\gamma'}
&3'\ar[dl]^{\alpha'}
\ar@(ru,rd)^{\epsilon'}\\
1'\ar[u]^{\beta'}
&4'\ar[u]_{\eta'}}\\
\alpha'\gamma'\beta'\alpha',\, \gamma'\beta'\alpha'\gamma',(\epsilon')^n,\epsilon'\gamma',\alpha'\epsilon',\alpha'\eta'.
\end{array}$$
Note that $B$ is a monomial algebra. By Zimmermann-Huisgen's result in \cite{zh92}, we know $B$ is $2$-syzygy finite. Then $A$ is syzygy-finite by Corollary \ref{cor-syzygy}.
}
\end{Bsp}

\section{Stable equivalences}
In this section, we shall prove that the extension dimension of the $i$-th syzygy module categories is invariant under stable equivalence for each $i\in \mathbb{N}$. We first recall some basic results about the stable equivalence of Artin algebras , as detailed in reference \cite{ars97,c21,Guo05}.

Let $A$ be an Artin algebra over a fixed commutative Artin ring $R$.
We denoted by $\stmodcat{A}$ the stable module category of $A$ modulo projective modules. The objects are the same as the objects of $A\modcat$, and for two modules $X, Y$ in $\stmodcat{A}$, their homomorphism set is $\underline{\Hom}_A(X,Y):=\Hom_A(X,Y)/\mathscr{P}(X,Y)$, where $\mathscr{P}(X,Y)$ is the subgroup of $\Hom_A(X,Y)$ consisting of the homomorphisms factorizing through a projective $A$-module. This category is usually called the \emph{stable module category} of $A$. Dually, We denoted by  $A\mbox{-}\overline{\mbox{mod}}$ the stable module category of $A$ modulo injective modules. Let $\tau_A$ be the Auslander-Reiten translation $\DTr$. Then $\tau_A:\stmodcat{A}\ra A\mbox{-}\overline{\mbox{mod}}$ be an equivalence as additive categories (see \cite[Chapler IV.1]{ars97}). Two algebras $A$ and $B$ are said to be \emph{stably equivalent} if the two stable categories $A\stmodcat$ and $B\stmodcat$ are equivalent as additive categories.

Next, suppose that $F: \stmodcat{A} \ra \stmodcat{B}$ is an stable equivalence. Then the following functor
$$F':=\tau_B\circ F\circ \tau_A^{-1}:A\mbox{-}\overline{\mbox{mod}}\lra B\mbox{-}\overline{\mbox{mod}}$$
is equivalent as additive categories. Moreover, there are one-to-one correspondences
$$F: A\modcat_{\mathscr{P}}\lra B\modcat_{\mathscr{P}}\quad \mbox{and}\quad
F': A\modcat_{\mathscr{I}}\lra  B\modcat_{\mathscr{I}},$$
where $A\modcat_{\mathscr{P}}$ (respectively, $A\modcat_{\mathscr{I}}$) stands for the full subcategory of $A\modcat$ consisting of modules without nonzero projective (respectively, injective) summands.
We also use $F$ (respectively, $F'$) to denote the induce map $A\modcat\ra B\modcat$ which takes projective modules (respectively, injective modules) to zero.

\begin{Def}{\rm (\cite[Definition 4.2]{c21})}{\rm
  An indecomposable $A\modcat$ module $S$ is called a node if it is neither projective nor
  injective, and there is an almost split sequence
  $0\to S\to P \to T \to 0$
  with $P$ a projective $A\modcat$ module.}
\end{Def}
\begin{Lem}{\rm (\cite[Chapter V, Theorem 3.3]{ars97})}
  A node module $S$ in $A\modcat$ must be simple.
\end{Lem}

\begin{Def}{\rm (\cite[Definition 4.3]{c21})}{\rm
 A node $S$ in $A\modcat$ is said to be an $F$-\emph{exceptional node} if $F(S)\not\cong F'(S)$. Let $\mathfrak{n}_{F}(A)$ be the set of isomorphism classes of $F$-exceptional nodes of $A$.}
\end{Def}
Since $\mathfrak{n}_{F}(A)$ is a subset of all simple modules, $\mathfrak{n}_{F}(A)$ is a finite set.

 \begin{Lem}{\rm (see \cite[Lemma 3.4]{ar78} or \cite[Chapter X.1.7, p. 340]{ars97})}
  Let $X$ be indecomposable, non-projective, non-injective, and not a node in $A\modcat$, then $F(X)\cong F'(X)$ .
 \end{Lem} Then $\mathfrak{n}_{F}(A)$ and the set of isomorphism classes
of indecomposable, non-projective, non-injective $A$-modules $U$ such that $F(U)\not\cong F'(U)$, coincide.

Let $F^{-1}:\stmodcat{B}\ra \stmodcat{A}$ be a quasi-inverse of $F$. Then we use $\mathfrak{n}_{F^{-1}}(B)$ to denote the the set of isomorphism classes of $F^{-1}$-exceptional nodes of $B$.

In the following, let $$\bigtriangleup_A
:=\mathfrak{n}_{F}(A)
\dot{\cup}(\mathscr{P}_A\setminus\mathscr{I}_A)
\quad\mbox{and}\quad \bigtriangledown_A
:=\mathfrak{n}_{F}(A)\dot{\cup}(\mathscr{I}_A\setminus\mathscr{P}_A),$$ where $\dot{\cup}$ stands for the disjoint union of sets; $\mathscr{P}_A$ and $\mathscr{I}_A$ stand for the set of isomorphism classes of indecomposable projective and injective $A$-modules, respectively. By $\bigtriangleup_A^c$ we mean the class of indecomposable, non-injective $A$-modules which do not belong to $\bigtriangleup_A$.

\begin{Rem}\label{rem-tri}
Each module $X\in A\modcat_{\mathscr{I}}$ admits a unique decomposition (up to isomorphism)
$$X\cong X_{\triangle}\oplus  X_c $$
with $X_{\triangle}\in \add(\bigtriangleup_A)$ and $X_c\in \add(\bigtriangleup^{c}_{A})$.
The module $X_{\triangle}$ and $X_c$ are called the $\triangle_{A}$\textit{-component} and $\bigtriangleup^{c}_{A}$\textit{-component} of $X$, respectively.
\end{Rem}

\begin{Lem}\label{one-to-one}{\rm (\cite[Lemma 4.10]{c21})}
There exist one-to-one correspondences
$$F:\bigtriangledown_A\lra \bigtriangledown_B
,\quad
F':\bigtriangleup_A\lra \bigtriangleup_B
\quad\mbox{and}\quad
F':\bigtriangleup_A^c\lra \bigtriangleup_B^c.$$
\end{Lem}

Recall that an exact sequence $0\ra X\lraf{f} Y\lraf{g} Z\ra 0$ in $A\modcat$ is called \emph{minimal} (\cite{MV1990}) if it has no a split exact sequence as a direct summand, that is, there does not exist isomorphisms $u$, $v$, $w$ such that the following diagram
$$\xymatrix{
0\ar[r]
&X\ar[r]^{f}\ar[d]^{u}
&Y\ar[r]^g\ar[d]^{v}
&Z\ar[r]\ar[d]^{w}
&0\\
0\ar[r]&X_1\oplus
X_2\ar[r]^{\text{
$\left(
\begin{smallmatrix}
f_1&0\\
0&f_2
\end{smallmatrix}\right)
$}}
&Y_1\oplus Y_2\ar[r]^{\text{
$\left(
\begin{smallmatrix}
g_1&0\\
0&g_2
\end{smallmatrix}\right)
$}}&Z_1\oplus Z_2\ar[r]&0
}$$
is commutative and has exact rows, where $Y_2\neq 0$ and $0\ra X_2\lraf{f_2} Y_2\lraf{g_2}Z_2\ra 0$ is split.

\begin{Lem}\label{minimal-1.1-1}
Let $A$ be an Artin algebra and $X,Y\in A\modcat$.
Let $u:X\oplus Y \ra X\oplus Y$ be a morphism with
$u=\left(
\begin{smallmatrix}
f_{11}&0\\
f_{21}&f_{22}
\end{smallmatrix}\right).$
If $f_{11}$ and $f_{22}$ are two isomorphisms, then $u$ is also an isomorphism.
\end{Lem}
\begin{proof}
Let $u'=\left(
\begin{smallmatrix}
f_{11}^{-1}&0\\
-f_{22}^{-1} f_{12}f_{11}^{-1} &f_{22}^{-1}
\end{smallmatrix}\right)$.
We can check that
$$uu'= \left(
\begin{smallmatrix}
f_{11}&0\\
f_{21}&f_{22}
\end{smallmatrix}\right)\circ \left(
\begin{smallmatrix}
f_{11}^{-1}&0\\
-f_{22}^{-1} f_{12}f_{11}^{-1} &f_{22}^{-1}
\end{smallmatrix}\right)=
\left(
\begin{smallmatrix}
1_{X}&0\\
0&1_{Y}
\end{smallmatrix}\right)=1_{X\oplus Y}$$
and
$$u'u= \left(
\begin{smallmatrix}
f_{11}^{-1}&0\\
-f_{22}^{-1} f_{12}f_{11}^{-1} &f_{22}^{-1}
\end{smallmatrix}\right)\circ
\left(
\begin{smallmatrix}
f_{11}&0\\
f_{21}&f_{22}
\end{smallmatrix}\right)
=
\left(
\begin{smallmatrix}
1_{X}&0\\
0&1_{Y}
\end{smallmatrix}\right)=1_{X\oplus Y}.$$
That is, $u$ is an isomorphism.
\end{proof}
\begin{Lem}\label{minimal-1.1}
Let $A$ be an Artin algebra. Let $0\ra E\oplus X\ra M \ra Z\ra 0$ be a short exact sequence in $A\modcat$,
where $E$ is a injective module. Then there exist isomorphisms $u$ and $v$ such that the following diagram with rows exact commutes.
$$\xymatrix{
0\ar[r]
& E\oplus X\ar[r]^{f}\ar[d]^{u}
&M\ar[r]^g\ar[d]^{v}
&Z\ar[r]\ar@{=}[d]
&0\\
0\ar[r]&E\oplus X\ar[r]^{\text{
$\left(
\begin{smallmatrix}
1_{E}&0\\
0&f_{22}
\end{smallmatrix}\right)
$}}
&E\oplus Y\ar[r]^{\text{
$\left(
\begin{smallmatrix}
0\\
g_2
\end{smallmatrix}\right)
$}}& Z\ar[r]&0.
}$$
\end{Lem}
\begin{proof}
Consider the following pushout
 \begin{gather}\label{pushout-diag}
\begin{split}
\xymatrix{
&0\ar[d] & 0 \ar[d]&&\\
 & E\ar@{=}[r]\ar[d]_{\text{
$\left(
\begin{smallmatrix}
1_{E}&0
\end{smallmatrix}\right)
$}}
& E\ar[d] & &\\
0\ar[r]& E\oplus X  \ar[r]^{f}\ar[d]_{\text{
$\left(
\begin{smallmatrix}
0\\1_{X}
\end{smallmatrix}\right)
$}}  & M\ar[r]^{g}\ar[d]& Z \ar[r]\ar @{=}[d]&0\\
0\ar[r]& X\ar[r]\ar[d] & Y \ar[r]\ar[d]&Z \ar[r]&0\\
&0&0& & }
\end{split}
\end{gather}
we can get that $M\cong E\oplus Y$ since $E$ is a injective module.
In the above diagram, replace $M$ with $E\oplus Y$, and set
 $f=\left(
\begin{smallmatrix}
f_{11}&f_{12}\\
f_{21}&f_{22}\\
\end{smallmatrix}\right)$ and
$g=\left(
\begin{smallmatrix}
g_{1}\\g_{2}
\end{smallmatrix}\right).$
Hence the commutative diagram (\ref{pushout-diag}) can be written as
\begin{gather}\label{6}
\begin{split}
\xymatrix{
&0\ar[d] & 0 \ar[d]&&\\
 & E\ar@{=}[r]\ar[d]_{\text{
$\left(
\begin{smallmatrix}
1_{E}&0
\end{smallmatrix}\right)
$}}
& E\ar[d]^{\text{
$\left(
\begin{smallmatrix}
1_{E}&0
\end{smallmatrix}\right)
$}} & &\\
0\ar[r]& E\oplus X  \ar[r]^{\text{
$\left(
\begin{smallmatrix}
f_{11}&f_{12}\\
f_{21}&f_{22}\\
\end{smallmatrix}\right)
$}}\ar[d]_{\text{
$\left(
\begin{smallmatrix}
0\\1_{X}
\end{smallmatrix}\right)
$}}  & E\oplus Y\ar[r]^{\text{
$\left(
\begin{smallmatrix}
g_{1}\\g_{2}
\end{smallmatrix}\right)
$}}\ar[d]_{\text{
$\left(
\begin{smallmatrix}
0\\1_{Y}
\end{smallmatrix}\right)
$}}
& Z \ar[r]\ar @{=}[d]&0\\
0\ar[r]& X\ar[r]^{f_{22}}\ar[d] & Y \ar[r]^{g_{2}}\ar[d]&Z \ar[r]&0\\
&0&0& & }
\end{split}
\end{gather}
Then
\begin{equation*}
\left\{
\begin{aligned}
\left(\begin{smallmatrix}
1_{E}&0
\end{smallmatrix}\right)
\left(\begin{smallmatrix}
f_{11}&f_{12}\\
f_{21}&f_{22}
\end{smallmatrix}\right)
&=\left(
\begin{smallmatrix}
1_{E}&0
\end{smallmatrix}\right)\\
\left(\begin{smallmatrix}
f_{11}&f_{12}\\
f_{21}&f_{22}
\end{smallmatrix}\right)
\left(
\begin{smallmatrix}
0\\1_{Y}
\end{smallmatrix}\right)
&=\left(
\begin{smallmatrix}
0\\1_{Y}
\end{smallmatrix}\right)
f_{22}\\
\left(
\begin{smallmatrix}
0\\1_{Y}
\end{smallmatrix}\right)
g_{2}=&\left(
\begin{smallmatrix}
g_{1}\\g_{2}
\end{smallmatrix}\right).
\end{aligned}
\right.
\end{equation*}
Thus $f_{11}=1_{E},f_{12}=0$ and $g_{1}=0.$
And then we have the following commutative diagram with rows exact
$$\xymatrix{
0\ar[r]
&E\oplus X\ar[r]^{\text{
$\left(
\begin{smallmatrix}
1_{E}&0\\
f_{21}&f_{22}
\end{smallmatrix}\right)
$}}\ar[d]^{\text{
$\left(
\begin{smallmatrix}
1_{E}&0\\
f_{21}&1_{X}
\end{smallmatrix}\right)
$}}
&E\oplus Y\ar[r]^{\text{
$\left(
\begin{smallmatrix}
0\\g_{2}
\end{smallmatrix}\right)
$}}\ar[d]^{\text{
$\left(
\begin{smallmatrix}
1_{E}&0\\
0&1_{Y}
\end{smallmatrix}\right)
$}}
&Z\ar@{=}[d]\ar[r]
&0\\
0\ar[r]&E\oplus
X\ar[r]_{\text{
$\left(
\begin{smallmatrix}
1_{E}&0\\
0&f_{22}
\end{smallmatrix}\right)
$}}
&E\oplus Y\ar[r]_{\text{
$\left(
\begin{smallmatrix}
0\\
g_2
\end{smallmatrix}\right)
$}}& Z\ar[r]&0
}$$
where $u=\left(
\begin{smallmatrix}
1_{E}&0\\
f_{21}&1_{Y}
\end{smallmatrix}\right)$ and
$v=\left(
\begin{smallmatrix}
1_{E}&0\\
0&1_{Y}
\end{smallmatrix}\right)$
are isomorphisms by Lemma \ref{minimal-1.1-1}.
\end{proof}

\begin{Lem}\label{minimal-1}
Let $0\ra X\ra Y\ra Z\ra 0$ be a minimal short exact sequence in $A\modcat$.
Then $X\in A\modcat_{\mathscr{I}}$ and $Z\in A\modcat_{\mathscr{P}}$.
\end{Lem}
\begin{proof}
If $X\notin A\modcat_{\mathscr{I}}$, then we have
$X=X'\oplus E$ for some injective module $0\neq E$ and $X'\in A\modcat_{\mathscr{I}}$.
By Lemma \ref{minimal-1.1}, we get that
$0\ra X\ra Y\ra Z\ra 0$ is not a minimal short exact sequence, a contradiction. Thus $X\in A\modcat_{\mathscr{I}}$.
Similarly, we also can get $Z\in A\modcat_{\mathscr{P}}$.
\end{proof}

The next lemma shows that the stable functor has certain ``exactness" property.
\begin{Lem}{\rm (\cite[Lemma 4.3]{zhang24})}\label{exact}
Let $Z$ be an $A$-module without nonzero projective summands, and let $$0\lra X\oplus X'\lra Y\oplus P\lraf{g}Z\lra 0$$ be a minimal short exact sequence in $A\modcat$ such that $X\in \add(\bigtriangleup_A^c)$, $X'\in \add(\bigtriangleup_A)$, $Y\in A\modcat_{\mathscr{P}}$ and $P\in\add(_AA)$.
Then there exists a minimal short exact sequence $$0\lra F(X)\oplus F'(X')\lra F(Y)\oplus Q\lraf{g'}F(Z)\lra 0$$ in $B\modcat$
such that $Q\in \add(_BB)$ and $g'=F(g)$ in $\stmodcat{B}$.
\end{Lem}

\begin{Lem}\label{exact1}
Let
\begin{align}\label{exact2}
0\lra X\lra Y\lra Z\lra 0
\end{align} be a short exact sequence in $A\modcat$.
Then there exists a short exact sequence $$0\lra V \lra F(Y)\oplus Q\lra F(Z)\lra 0$$ in $B\modcat$
with $Q\in \add(_BB)$ and $V\in [F(X)\oplus F'(W)]_{1}$, where $W$ is an $A$-module with $\add(W)=\add(\bigtriangleup_A)$.
\end{Lem}
\begin{proof}
We can decompose the short exact sequence (\ref{exact2}) as the direct sums of the following two short exact sequences
\begin{align}\label{exact3}
0\lra X_1\lra Y_1\lra Z_1\lra 0,
\end{align}
\begin{align}\label{exact4}
0\lra X_2\lra Y_2\lra Z_2\lra 0
\end{align}
in $A\modcat$, namely there are isomorphisms $u$, $v$, $w$ with the following commutative diagram
\begin{align}\label{2row}
\xymatrix@C=0.5em@R=0.5em{
0\ar[rr]
&&X\ar[rr]\ar[dd]^{u}
&&Y\ar[rr]\ar[dd]^{v}
&&Z\ar[rr]\ar[dd]^{w}
&&0\\
&&&&&&&&\\
0\ar[rr]
&&X_1\oplus
X_2\ar[rr]
&&Y_1\oplus Y_2\ar[rr]
&&Z_1\oplus Z_2\ar[rr]
&&0,
}\end{align}
in $A\modcat$ such that (\ref{exact3}) is minimal and (\ref{exact4}) is split.
By Lemma \ref{minimal-1}, we have
$X_{1}\in A\modcat_{\mathscr{I}}$ and $Z_{1}\in A\modcat_{\mathscr{P}}$.
By Remark \ref{rem-tri}, we write
$X_{1}\cong X_{1}'\oplus X_{2}''$, where $X_{1}'\in \add(\bigtriangleup_A^c)$ and $X_{2}''\in \add(\bigtriangleup_A)$
, $Y_{1}\cong Y_{1}'\oplus P$ where
$Y'_{1}\in A\modcat_{\mathscr{P}}$ and $ P\in \pmodcat{A}$.
That is, we have the following exact sequence
\begin{align}\label{exact5}
0\lra X_{1}'\oplus X_{2}''\lra  Y_{1}'\oplus P\lra Z_1\lra 0.
\end{align}
And by Lemma \ref{exact} and the exact sequence (\ref{exact5}),
we get the following exact sequence
\begin{align}\label{exact6}
0\lra F(X_{1}')\oplus F'(X_{2}'')\lra  F(Y_{1}')\oplus Q\lra F(Z_1)\lra 0.
\end{align}
where $Q\in \add(_BB)$.
Applying the functor $F$ to the split short exact sequence (\ref{exact4}), we can get the following exact sequence
\begin{align}\label{exact7}
0\lra F(X_2)\lra F(Y_2)\lra F(Z_2)\lra 0.
\end{align}
By exact sequence (\ref{exact6}) and (\ref{exact7}), we can get the following exact sequence
\begin{align}\label{exact8}
0\lra F(X_{1}')\oplus F'(X_{2}'') \oplus F(X_2)\lra  F(Y_{1}')\oplus Q \oplus F(Y_2)\lra F(Z_1)\oplus F(Z_2)\lra 0.
\end{align}
Then the exact sequence (\ref{exact8}) can be written as
\begin{align}\label{exact9}
0\lra V\lra  F(Y)\oplus Q \lra F(Z)\lra 0,
\end{align}
where $$V:=F(X_{1}')\oplus F'(X_{2}'') \oplus F(X_2)\cong F(X_{1}'\oplus X_{2})\oplus F'(X_{2}'')\in [F(X)\oplus F'(W)]_{1},$$
$F(Y)=F(Y_{1}\oplus Y_{2})=F(Y_{1}'\oplus P\oplus Y_{2})=F(Y_{1}')\oplus F(Y_{2}),$
and $F(Z)=F(Z_{1}\oplus Z_{2})\cong F(Z_1)\oplus F(Z_2)$.
\end{proof}
\begin{Lem}{\rm (\cite[Lemma 4.14]{c21})}\label{chen-rigid}
  Let $X\in A\modcat$ and $n$ a positive integer.
  Then
  $$F(\Omega_{A}^{n}(X))\oplus \bigoplus_{j=1}^{n}\Omega_{B}^{n-j}(F'(\Omega_{A}^{j}(X)_{\triangle}))
  \cong \Omega_{B}^{n}(F(X))\oplus  \bigoplus_{j=1}^{n}\Omega_{B}^{n-j}(F(\Omega_{A}^{j}(X)_{\triangle}))$$
  where $\Omega^{j}_{A}(X)_{\triangle}$ stands for the $\triangle_{A}$-component of the $A$-module $\Omega^{j}_{A}(X)$.
\end{Lem}
Zhang-Zheng have the following result.
\begin{Lem}{\rm (\cite[Theorem 1.2]{zhang24})}\label{dimensionA=B}
Let $A$ and $B$ be stably equivalent Artin algebras. Then $$\ed A\modcat=\ed B\modcat.$$
\end{Lem}

\begin{Theo}\label{st-thm}
Let $A$ and $B$ be stably equivalent Artin algebras. Then $$\ed\Omega^{n}(A\modcat)=\ed\Omega^{n}(B\modcat)$$ for each nonnegative integer $n$.
\end{Theo}
\begin{proof}
By Lemma \ref{dimensionA=B}, we only need consider the case $n>0$.
For a module $\Omega^{n}_{B}(M)\in \Omega^{n}(B\modcat)$, where $M\in B\modcat$.
We can set $M\cong M'\oplus M''$ where $M'\in B\modcat_{\mathscr{P}}$ and $M''\in \pmodcat{B}$.
Since $F: A\modcat_{\mathscr{P}}\ra B\modcat_{\mathscr{P}}$ is an one-to-one correspondence, there exists a module $X\in A\modcat_{\mathscr{P}}$ such that $F(X)\cong M'$ as $B$-modules.
Then we have
\begin{equation}\label{exact10.1}
\xymatrix@C=1.5em@R=0.1em{
\Omega^{n}_{B}(M)\cong \Omega^{n}_{B}(M'\oplus M'')\cong \Omega^{n}_{B}(M')\cong \Omega^{n}_{B}(F(X)).}
\end{equation}

Suppose that $\ed\Omega^{n}(A\modcat)=m.$
Then we can set
$\ed\Omega^{n}(A\modcat)\subseteq [T]_{m+1}$ for some module $T\in A\modcat$ by Definition \ref{def-subcat-extension-dim}.
For a module $\Omega^{n}_{A}(X)\in \Omega^{n}(A\modcat)\subseteq [T]_{m+1}$, we have the following short exact sequences
\begin{equation}\label{exact10}
\begin{cases}
\xymatrix@C=1.5em@R=0.1em{
0\ar[r]& X_{1}\ar[r] & \Omega^{n}_{A}(X)\oplus X'\ar[r]& Y_{1}\ar[r]&0,\\
0\ar[r]& X_{2}\ar[r] & Y_{1}\oplus Y_{1}'\ar[r]& Y_{2}\ar[r]&0,\\
&&\vdots&&\\
0\ar[r]& X_{m}\ar[r] & Y_{m-1}\oplus Y_{m-1}'\ar[r]& Y_{m}\ar[r]&0.
}
\end{cases}
\end{equation}
where $X_i\in [T]_1$ and $Y_i\in [T]_{m+1-i}$ for $1\leqslant i\leqslant m$.
Let $W$ be an $A$-module with
$\add(W)=\add(\bigtriangleup_A).$
By Lemma \ref{exact1} and the above short exact sequences (\ref{exact10}), we can
get the following short exact sequences
\begin{equation}\label{exact11}
\begin{cases}
\xymatrix@C=0.8em@R=0.1em{
0\ar[r]& V_{1}\ar[r] & F(\Omega^{n}_{A}(X))\oplus F(X')\oplus Q_{1}\ar[r]& F(Y_{1})\ar[r]&0,\\
0\ar[r]& V_{2}\ar[r] & F(Y_{1})\oplus F(Y_{1}')\oplus Q_{2}\ar[r]& F(Y_{2})\ar[r]&0,\\
&&\vdots&&\\
0\ar[r]& V_{m-1}\ar[r] & F(Y_{m-2})\oplus F(Y_{m-2}')\oplus Q_{m-1}\ar[r]& F(Y_{m-1})\ar[r]&0,\\
0\ar[r]& V_{m}\ar[r] & F(Y_{m-1})\oplus F(Y_{m-1}')\oplus Q_{m}\ar[r]& F(Y_{m})\ar[r]&0,
}
\end{cases}
\end{equation}
where $V_{i}\in[F(T)\oplus F'(W)]_{1}$ for $1\leqslant i\leqslant m$ and $Y_{m}\in[F(T)]_{1}$.
Thus
\begin{equation}\label{exact10.2}
\xymatrix@C=1.5em@R=0.1em{
F(\Omega^{n}_{A}(X))\in[F(T)\oplus F'(W)]_{m+1}.}
\end{equation}
Denoted by $\Omega^{j}_{A}(X)_{\triangle}$ stands for the $\triangle_{A}$-component of the $A$-module $\Omega^{j}_{A}(X)$ for $1\leqslant j\leqslant n$. Then we have
 \begin{align*}
 & \Omega_{B}^{n}(M)\oplus  \bigoplus_{j=1}^{n}\Omega_{B}^{n-j}(F(\Omega_{A}^{j}(X)_{\triangle}))\\
\cong & \Omega_{B}^{n}(F(X))\oplus  \bigoplus_{j=1}^{n}\Omega_{B}^{n-j}(F(\Omega_{A}^{j}(X)_{\triangle}))\ \ \ \ \text{(by (\ref{exact10.1}))}\\
\cong&F(\Omega^{n}_{A}(X))\oplus  \bigoplus_{j=1}^{n}\Omega_{B}^{n-j}(F'(\Omega_{A}^{j}(X)_{\triangle}))\ \ \ \ \text{(by Lemma \ref{chen-rigid})}\\
\subseteq& [F(T)\oplus F'(W)]_{m+1}\oplus [\bigoplus_{j=1}^{n}\Omega_{B}^{n-j}(F'(\Omega_{A}^{j}(X)_{\triangle}))]_{1}\ \ \ \ \text{(by (\ref{exact10.2}))}\\
\subseteq& [F(T)\oplus F'(W)]_{m+1}\oplus [\bigoplus_{j=1}^{n}\Omega_{B}^{n-j}(F'(W))]_{1}\\
\subseteq& [F(T)\oplus \bigoplus_{j=1}^{n}\Omega_{B}^{n-j}(F'(W))]_{m+1}.\ \ \ \ \text{(by Lemma \ref{sum-subseteq-1})}
\end{align*}
Note that
$\Omega^{n}(B\modcat)=\{K\oplus Q\in B\modcat \mid K\cong\Omega^{n}(M) \text{ for some }B\text{-module } M \text{ and } Q\in \pmodcat{B}\}.$
Then we have
$$\Omega^{n}(B\modcat)\subseteq [F(T)\oplus \bigoplus_{j=1}^{n}\Omega_{B}^{n-j}(F'(W))\oplus B]_{m+1}.$$
By Definition \ref{def-subcat-extension-dim}, we have $$\ed \Omega^{n}(B\modcat) \leqslant m=\ed \Omega^{n}(A\modcat).$$
Similarly, we also can get
$$\ed \Omega^{n}(A\modcat) \leqslant \ed \Omega^{n}(B\modcat).$$
Moreover, we get $$\ed \Omega^{n}(A\modcat) = \ed \Omega^{n}(B\modcat).$$
This finishes the proof.
\end{proof}

\begin{Koro}
If $A$ and $B$ are stably equivalent, then $A$ is $n$-syzygy-finite if and only if so $B$ does for each integer $n\in\mathbb{N}$.
\end{Koro}
\begin{proof}
Note that $A$ is $n$-syzygy-finite if and only if $\Omega^{n}(A\modcat)$ is representation-finite if and only if $\ed\Omega^{n}(A\modcat)=0$ by Definition \ref{def-syzygy-modcat}. Similarly, $B$ is $n$-syzygy-finite if and only if $\Omega^{n}(B\modcat)$ is representation-finite if and only if $\ed\Omega^{n}(B\modcat)=0$. By Theorem \ref{st-thm}, $\ed\Omega^{n}(A\modcat)=0$ if and only if $\ed\Omega^{n}(B\modcat)=0$. Then the Corollary follows.
\end{proof}

To illustrate Theorem \ref{st-thm}, we give the following example.
\begin{Bsp}{\rm (\cite[Example 4.8]{zhang24})
Let $k$ be a fixed field, $A=kQ_A/I$, where $Q_A$ is the quiver
$$\xymatrix@R=10pt{
&&1\ar[d]^{\beta}\ar@(ur,ul)_{\gamma}&&&&\\
2\ar[r]^{\alpha_2}
&3\ar[r]^{\alpha_3}
&4\ar[r]^{\alpha_4}
&5\ar[r]^{\alpha_5}
&6\ar[r]^{\alpha_6}
&\cdots \ar[r]^{\alpha_{n-1}}
&n
}$$
and $I$ is generated by $\{\gamma^2, \beta\gamma\}$ with $n\geqslant 6$.
By \cite[Lemma 1]{MV1980}, we know that $S(1)$ is a unique node of $A$. It follows from \cite[Theorem 2.10]{MV1980} that $A$ is stably equivalent to the path algebra $B$ given by the following quiver $Q_B$:
$$\xymatrix@R=10pt{
&&1'&&&&\\
&&1\ar[d]^{\beta}\ar[u]_{\delta}&&&&\\
2\ar[r]^{\alpha_2}
&3\ar[r]^{\alpha_3}
&4\ar[r]^{\alpha_4}
&5\ar[r]^{\alpha_5}
&6\ar[r]^{\alpha_6}
&\cdots \ar[r]^{\alpha_{n-1}}
&n.}$$
Since the underlying graph of $Q_B$ is not Dynkin, $B$ is representation-infinite and $\ed(B\modcat)=1$ by Lemmas \ref{ext-lem}(1) and \ref{ext-lem}(3). Due to $\gldim(B)=1$, $\Omega(B\modcat)\subseteq\pmodcat{B}$ and $\ed \Omega(B\modcat)=0$. In particular, $\ed \Omega^i(B\modcat)=0$ for each $i\geqslant 1$. By Theorem \ref{st-thm}, we have $\ed(A\modcat)=\ed(B\modcat)=1$, and
$\ed \Omega^i(A\modcat)=\ed \Omega^i(B\modcat)=0$ for each $i\geqslant 1$.}
\end{Bsp}

\begin{Bsp}{\rm
Let $k$ be a fixed field and $A$ be a $k$-algebra given by quiver with relations
$$\begin{array}{ccc}
\xymatrix@R=8pt{
1\ar[rd]^{\alpha}&&4&&\\
&3\ar[ru]^{\delta}\ar[rd]_{\beta}&&7\ar[rd]^{\nu}&\\
2\ar[ru]_{\gamma}&&5\ar[ru]^{\mu}\ar[rd]_{\eta}&&9\\
&&&8\ar[ru]_{\theta}&
}\\
\beta\alpha=\delta\gamma=\mu\beta=\nu\beta=0,\,\nu\mu=\theta\eta.
\end{array}$$
We claim that $A$ is representation finite. Indee, it follows from \cite[p. 430]{MV1980} that $A$ and $B_1\times B_2$ are stably equivalent, where $B_1$ and $B_2$ are $k$-algebras given by quivers with relations, respectively:
$$\begin{array}{ccc}
\xymatrix@R=8pt{
1\ar[rd]^{\alpha}&&4\\
&3\ar[ru]^{\delta}\ar[rd]_{\beta}&\\
2\ar[ru]_{\gamma}&&5\\
}
&\quad&
\xymatrix@R=8pt{
&7\ar[rd]^{\nu}&\\
6\ar[ru]^{\mu}\ar[rd]_{\eta}&&9\\
&8\ar[ru]_{\theta}&
}\\
\beta\alpha=\delta\gamma=0,&\quad&\nu\mu=\theta\eta=0.
\end{array}$$
By calculation, $B_1$ and $B_2$ are representation finite. And then $\ed (B_1)=\ed(B_2)=0$ by Lemma \ref{ext-lem}(1).
By Theorem \ref{st-thm}, $\ed(A)=0$. And then $A$ is representation finite.
}
\end{Bsp}

Now, suppose that $A$ and $B$ are stably equivalent Artin algebras without nodes. Then $\mathfrak{n}_{F}(A)=\varnothing$, $\mathfrak{n}_{F^{-1}}(B)=\varnothing$, and $F': \mathscr{P}_A\setminus\mathscr{I}_A\ra \mathscr{P}_B\setminus\mathscr{I}_B$ is bijective by Lemma \ref{one-to-one}. Moreover, Lemma \ref{exact} can be specified as follows.
\begin{Lem}\label{lemma-exact-node}
Let $A$ and $B$ be stably equivalent Artin algebras without nodes, and let
\begin{align*}
0\lra X\oplus P_1\lra P\lra Z\lra 0
\end{align*} be a minimal short exact sequence in $A\modcat$
where $X,Z\in A\modcat_{\mathscr{P}}$, $P_1\in \mathscr{P}_A\setminus\mathscr{I}_A$ and $P \in \pmodcat{A}$. Then
there is a minimal exact sequence
$$0\lra F(X)\oplus F'(P_1)\lra Q\lra F(Z)\lra 0$$
in $B\modcat$ where $Q \in \pmodcat{B}$.
\end{Lem}

\begin{Lem}\label{Lem-node-infinite-tran}
Let $A$ and $B$ be stably equivalent Artin algebras without nodes.
If $X\in\Omega^{\infty}(A\modcat)$, then $F(X)\oplus Q\in\Omega^{\infty}(B\modcat)$ for some $Q\in \pmodcat{B}$.
\end{Lem}
\begin{proof}
It follows from $X\in\Omega^{\infty}(A\modcat)$ that there is an exact sequence
$$0\lra X\lra P^0\lraf{f^0} P^1\lraf{f^1} P^2\lraf{f^2} \cdots,$$
where $P^i\in \pmodcat{A}$.
Let $K^i$ be the kernel of $f^i$ for $i\geqslant 0$. Then $K^0=X$ and we have short exact sequences
\begin{align}\label{exact-infinite}
0\lra K^i\lra P^i\lra K^{i+1}\lra 0, \text{ for } i\geqslant 0.
\end{align}
We can decompose the short exact sequences (\ref{exact-infinite}) as the direct sums of the following two short exact sequences
\begin{align}
0\lra U_1^i\oplus U_2^i\lra P_1^i\lra C^{i+1}\lra 0, \text{ for } i\geqslant 0, \mbox{ and}\label{exact-infinite-2}\\
0\lra V^i\lra P_2^i\lra D^{i+1}\lra 0, \text{ for } i\geqslant 0\label{exact-infinite-3}
\end{align}
in $A\modcat$ such that the short exact sequences (\ref{exact-infinite-2}) are minimal, the short exact sequences (\ref{exact-infinite-3}) are split, and $U_1^i\oplus U_2^i\oplus V^i\cong K^i$, $P_1^i\oplus P_2^i\cong P^i$, $C^{i+1}\oplus D^{i+1}\cong K^{i+1}$, $U_1^i\in A\modcat_{\mathscr{P}}$ and $U_2^i\in \pmodcat{A}$. By Lemma \ref{lemma-exact-node} and the minimal short exact sequences (\ref{exact-infinite-2}), we get the following minimal short exact sequences
\begin{align}\label{exact-infinite-4}
0\lra F(U_1^i)\oplus F'(U_2^i)\lra Q_1^i\lra F(C^{i+1})\lra 0, \text{ for } i\geqslant 0,
\end{align}
where $Q_1^i\in \pmodcat{B}$. Since the short exact sequences (\ref{exact-infinite-3}) are split, $V^i$ and $D^{i+1}$ are projrctive for $i\geqslant 0$. Note that $X=K^0\cong U_1^0\oplus U_2^0\oplus V^0$, $U_2^0\in \pmodcat{A}$, $K^i\cong C^i\oplus D^i\cong U_1^i\oplus U_2^i\oplus V^i$, and $U_2^i\in \pmodcat{A}$ for $i\geqslant 1$. Thus $F(X) \cong F(U_1^0)$ and $F(K^i)\cong F(C^i)\cong F(U_1^i)$ as $B$-modules. By the exact sequences (\ref{exact-infinite-4}), we have a long exact sequence
$$0\lra F(X)\oplus F'(U_2^0)\lra Q_1^0\oplus F'(U_2^1)\lra  Q_1^1\oplus F'(U_2^2)\lra Q_1^2 \oplus F'(U_2^3)\lra \cdots.$$
By Lemma \ref{one-to-one}, $F': \mathscr{P}_A\setminus\mathscr{I}_A\ra \mathscr{P}_B\setminus\mathscr{I}_B$ is bijective. Then $F'(U_2^i)\in \pmodcat{B}$ for $i\geqslant 0$. Thus $F(X)\oplus F'(U_2^0)\in \Omega^{\infty}(B\modcat)$.
\end{proof}

\begin{Theo}\label{st-thm-node}
Let $A$ and $B$ be stably equivalent Artin algebras without nodes. Then $$\ed\Omega^{\infty}(A\modcat)=\ed\Omega^{\infty}(B\modcat).$$
\end{Theo}
\begin{proof}
Let $M\in \Omega^{{\infty}}(B\modcat)$. For the stable functor $F^{-1}:\stmodcat{B}\ra \stmodcat{A}$, we get that $F^{-1}(M)\oplus P\in \Omega^{{\infty}}(A\modcat)$ for some projective $A$-module $P$.

Suppose that $\ed\Omega^{\infty}(A\modcat)=m.$
Then we can set
$\ed\Omega^{\infty}(A\modcat)\subseteq [T]_{m+1}$ for some $A$-module $T$ by Definition \ref{def-subcat-extension-dim}.
Define $X:=F^{-1}(M)\oplus P$. It follows from $X\in \Omega^{{\infty}}(A\modcat)$ that there are the short exact sequences
\begin{equation}\label{exact-infinite-1}
\begin{cases}
\xymatrix@C=1.5em@R=0.1em{
0\ar[r]& X_{1}\ar[r] & X\oplus X'\ar[r]& Y_{1}\ar[r]&0,\\
0\ar[r]& X_{2}\ar[r] & Y_{1}\oplus Y_{1}'\ar[r]& Y_{2}\ar[r]&0,\\
&&\vdots&&\\
0\ar[r]& X_{m}\ar[r] & Y_{m-1}\oplus Y_{m-1}'\ar[r]& Y_{m}\ar[r]&0.
}
\end{cases}
\end{equation}
where $X_i\in [T]_1$ and $Y_i\in [T]_{m+1-i}$ for $1\leqslant i\leqslant m$.
Let $W$ be an $A$-module with $\add(W)=\add(\bigtriangleup_A).$
Similar to the proof process in Theorem \ref{st-thm}, we deduce
$F(X)\in[F(T)\oplus F'(W)]_{m+1}$. Thus $M\in[F(T)\oplus F'(W)\oplus B]_{m+1}$ and
$$\Omega^{\infty}(B\modcat)\subseteq[F(T)\oplus F'(W)\oplus B]_{m+1}.$$
By Definition \ref{def-subcat-extension-dim}, we have $$\ed \Omega^{\infty}(B\modcat) \leqslant m=\ed \Omega^{\infty}(A\modcat).$$
Similarly, we also can get
$$\ed \Omega^{\infty}(A\modcat) \leqslant \ed \Omega^{\infty}(B\modcat).$$
Moreover, we get $$\ed \Omega^{\infty}(A\modcat) = \ed \Omega^{\infty}(B\modcat).$$
\end{proof}

\section{Separable equivalences}

In this section, we will prove that the extension dimension of the $i$-th syzygy module categories is an invariant under separable equivalence, for each $i\in \mathbb{N}$ or $i=\infty$. We begin by recalling the definition of separable equivalence of Artin algebras (see \cite{Linckelmann11}), which includes the derived equivalence of self-injective algebras, Morita equivalence,
stable equivalence of Morita type (\cite{Pea2017}) and  singular equivalences of Morita type (\cite{zhou2013}).

\begin{Def}{\rm (\cite{Linckelmann11})}
\label{def-4.4}
{\rm Two Artin algebras $A$ and $\Gamma$ are called {\bf separably equivalent} if there exist $_{B}M_{A}$
and $_{B}N_{A}$ such that
\begin{itemize}
\item[(1)] $M$ and $N$ are both finitely generated projective as one sided modules;
\item[(2)] $M\otimes_{A}N\cong B\oplus U$ as a $(B,B)$-bimodule for some $_{B}U_{B}$;
\item[(3)] $N\otimes_{B}M\cong A\oplus V$ as a $(A,A)$-bimodule for some $_{A}V_{A}$.
\end{itemize}}
\end{Def}

\begin{Lem}{\rm (\cite[Lemma 2.4]{zheng2020})} \label{exact-functor}
Let $A$ and $B$ be Artin algebras. If the functor $F:A\modcat\to B\modcat$ is exact functor, then
we have $F([T]_{n})\subseteq [F(T)]_{n}$ for each module $T\in A\modcat$ and each integer $n\geqslant 0.$
\end{Lem}
\begin{Lem}{\rm (\cite[Theorem 4.5]{zheng2020})}\label{separable-theorem1-1}
Let $A$ and $B$ be Artin algebras. If they are separably equivalent,
then $\ed A\modcat=\ed B\modcat$.
\end{Lem}
\begin{Theo}\label{sep-thm}
Let $A$ and $B$ be Artin algebras. If they are separably equivalent,
then $$\ed \Omega^{i}(A\modcat)=\ed \Omega^{i}(B\modcat)$$ for each $i\in \mathbb{N}\cup \{\infty\}.$
\end{Theo}
\begin{proof}
Let $M$ and $N$ be as in Definition \ref{def-4.4}.
The case $i=0$ can be seen Lemma \ref{separable-theorem1-1}.
Now consider the case $i>0$ or $i=\infty$. Let $\ed\Omega^{i}(B\modcat)=n$.
Then there exists $T\in B\modcat$
such that $\Omega^{i}(B\modcat)\subseteq[T]_{n+1}$by Definition \ref{def-subcat-extension-dim}.
Let $X\in \Omega^{i}(A\modcat)$.
Consider the following exact sequences
\begin{align}
0\ra X\ra P^0\ra P^1\ra P^2\ra \cdots,&\quad \text{if } i=\infty\label{exact-x1}\\
0\ra X\ra P^0\ra P^1\ra  \cdots\ra P^{i-1},&\quad \text{if } i<\infty\label{exact-x2}
\end{align}
where $P^{j}\in \pmodcat{A}$ for each $j$.
Applying the functor $_{B}M\otimes_{A}-$
to the above exact sequence (\ref{exact-x1}) or (\ref{exact-x2}), we can get the following exact
sequences
\begin{align}
0\ra M\otimes_{A}X\ra M\otimes_{A}P^0\ra M\otimes_{A}P^1\ra M\otimes_{A}P^2\ra \cdots,&\quad \text{if } i=\infty\label{exact-x3}\\
0\ra M\otimes_{A}X\ra M\otimes_{A}P^0\ra M\otimes_{A}P^1\ra  \cdots\ra M\otimes_{A}P^{i-1},&\quad \text{if } i<\infty\label{exact-x4}
\end{align}
in $B\modcat$, where  $M\otimes_{A}P^j\in\pmodcat{B}$
since $_{A}P^{j}$ and $_{B}M$ are projective. Moreover, we have
$$M\otimes_{A}X\in \Omega^{i}(B\modcat)\subseteq [T]_{n+1}$$
in $B\modcat$.
Since $N_{B}$ is projective, we obtain that the
functor $N\otimes_{B}-:B\modcat \ra A\modcat$ is exact. By Lemma \ref{exact-functor}, we have
$$N\otimes_{B}M\otimes_{A}X\in  [N\otimes_{B}T]_{n+1}.$$
By Definition \ref{def-4.4}(3), $X\in  [N\otimes_{B}T]_{n+1}$. Then $\Omega^{i}(A\modcat)\subseteq [N\otimes_{B}T]_{n+1}$ and $\ed \Omega^{i}(A\modcat)\leqslant n=\ed \Omega^{i}(B\modcat)$. Symmetrically, we have $\ed \Omega^{i}(B\modcat) \leqslant \ed \Omega^{i}(A\modcat)$.
Moreover, we get $\ed \Omega^{i}(B\modcat)= \ed \Omega^{i}(A\modcat)$.
\end{proof}

\begin{Koro}
Let $G$ be a finite group and $k$ a field of characteristic $p>0$. If $P$ is a Sylow $p$-subgroup of $G$,
then $$\ed \Omega^{i}(kP\modcat)=\ed \Omega^{i}(kG\modcat)$$ for each $i\in \mathbb{N}\cup\{\infty\}.$
\end{Koro}
\begin{proof}
It follows from \cite{Pea2017} that $kP$ and $kG$ are separably equivalent. Then the statement of the Corollary follows from Theorem \ref{sep-thm}.
\end{proof}

\begin{Koro}\label{stable-Morita1}
Let $A$ and $B$ be Artin algebras. If they are stably equivalent of Morita type or singular equivalences of Morita type,
then $$\ed \Omega^{i}(A\modcat)=\ed \Omega^{i}(B\modcat)$$ for each $i\in \mathbb{N}\cup\{\infty\}.$
\end{Koro}

Recall that a derived equivalence $F$ between finite dimensional algebras $A$ and $B$ with a quasi-inverse $G$ is called \emph{almost $\nu$-stable} \cite{hx10} if the associated radical tilting complexes $\cpx{T}$ over $A$ and $\cpx{\bar{T}}$ over $B$ are of the form
$$ \cpx{T}: 0\lra T^{-n}\lra \cdots T^{-1}\lra T^0\lra 0 \; \mbox{ and } \,
\cpx{\bar{T}}: 0\lra \bar{T}^{0}\lra \bar{T}^1\lra \cdots \lra \bar{T}^{n}\lra 0,$$
respectively, such that $\add(\bigoplus_{i=1}^n T^{-i})=\add(\nu_A(\bigoplus_{i=1}^n T^{-i}))$ and $\add(\bigoplus_{i=1}^n \bar{T}^{i}) = \add(\nu_B(\bigoplus_{i=1}^n \bar{T}^{i}))$, where $\nu$ is the Nakayama functor. By \cite[Theorem 1.1(2)]{hx10}, almost $\nu$-stable derived equivalences induce special stable equivalences, namely stable equivalences of Morita type. Thus we have the following consequence of Corollary \ref{stable-Morita1}.

\begin{Koro}\label{nu-s}
Let $A$ and $B$ be almost $\nu$-stable derived equivalent finite dimensional algebras. Then $$\ed\Omega^{i}(A\modcat)=\ed\Omega^{i}(B\modcat)$$ for each $i\in \mathbb{N}\cup\{\infty\}.$
\end{Koro}

Recall that given a finite dimensional algebra $A$ over a filed $k$, $A\ltimes D(A)$, the trivial extension of $A$ by $D(A)$ is the $k$-algebra whose underlying $k$-space is $A\oplus D(A)$, with multiplication given by

$$ (a,f)(b, g)= (ab,fb+ag)$$ for $a,b\in A$, and $f,g\in D(A)$, where $D:=\Hom_{k}(-,k)$. It is known that $A\ltimes D(A)$ is always symmetric, and therefore it is selfinjective.

\begin{Koro}
Let $A$ and $B$ be derived equivalent finite dimensional algebras. Then $$\ed\Omega^{i}(A\ltimes D(A)\modcat)=\ed\Omega^{i}(B\ltimes D(B)\modcat),$$
for each $i\in \mathbb{N}\cup\{\infty\}.$
\end{Koro}
\begin{proof}
By a result of Rickard (see \cite[Theorem 3.1]{r89}), which says that any derived equivalence between two algebras induces a derived equivalence between their trivial extension algebras, we obtain that $A\ltimes D(A)$ and $B\ltimes D(B)$ are derived equivalent.  It follows from \cite[Proposition 3.8]{hx10} that every derived equivalence between two selfinjective algebras induces an almost $\nu$-stable derived equivalence. Thus we have $\ed\Omega^{i}(A\ltimes D(A)\modcat)=\ed\Omega^{i}(B\ltimes D(B)\modcat)$ by Corollary \ref{nu-s}.
\end{proof}

\begin{Koro}\label{stable-Morita2}
Let $A$ be a self-injective, then for
any $A$-module $X$ and $n\in\mathbb{Z}$, we have
$$\ed \Omega^{i}(\End_{A}(A\oplus X)\modcat)=\ed \Omega^{i}(\End_{A}(A\oplus \Omega^{n}(X))\modcat)$$
and
$$\ed \Omega^{i}(\End_{A}(A\oplus X)\modcat)=\ed \Omega^{i}(\End_{A}(A\oplus \tau_{A}^{n}(X))\modcat)$$
for each integer $i\in \mathbb{N}\cup\{\infty\},$ where $\tau$ stands for the Auslander–Reiten translation.
\end{Koro}
\begin{proof}
By \cite[Corollary 3.4]{liu2007}, we know that $\End_{A}(A\oplus X)$ and $\End_{A}(A\oplus \Omega^{n}(X))$ are stably equivalent of Morita type,
and $\End_{A}(A\oplus X)$ and $\End_{A}(A\oplus \tau_{A}^{n}(X))$ are stably equivalent of Morita type,
for each $n\in\mathbb{Z}$. Then the statement follows from Corollary \ref{stable-Morita1}.
\end{proof}

The following example taken from \cite[Example 3]{lx06}.
\begin{Bsp}{\rm
Let $A$ be a finite dimensional algebra over a field $k$ given by quiver with relations:
$$\begin{array}{ccc}
\xymatrix{
a \ar@<0.5ex>[r]^{\rho}
& k\ar@<0.5ex>[l]^{\rho'}
\ar@<0.5ex>[r]^{\delta}
&b\ar@<0.5ex>[l]^{\delta'}\\
&c\ar[u]^{\kappa}&
}\\
\rho'\rho=\delta\delta'=0,\,
\rho'\kappa=\delta\kappa=0,\,
\delta'\delta\rho\rho'=\rho\rho'\delta'\delta,
\end{array}$$
and $B$ be a $k$-algebra given by the quiver with relations:
$$\begin{array}{ccc}
\xymatrix{
1 \ar@<0.5ex>[rr]^{\alpha}
\ar@<0.5ex>[rd]^{\eta}&& 2\ar@<0.5ex>[ll]^{\delta}
\ar@<0.5ex>[ld]^{\beta}\\
&k\ar@<0.5ex>[lu]^{\gamma}
\ar@<0.5ex>[ru]^{\xi}&\\
&3\ar[u]^{\kappa}&
}\\
\delta\alpha=\gamma\beta,\xi\beta=\alpha\delta,\eta\gamma=\beta\xi,\,
\beta\alpha=\gamma\beta=\alpha\beta=\eta\delta=\xi\eta=\delta\xi=0,\,\gamma\kappa=\xi\kappa=0.
\end{array}$$
It follows from \cite[Example 3]{lx06} that $A$ and $B$ are stably equivalent of Morita type. By Corollary \ref{stable-Morita1}, we have
$\ed \Omega^{i}(A\modcat)=\ed \Omega^{i}(B\modcat)$ for each $i\in \mathbb{N}\cup\{\infty\}.$

}\end{Bsp}

\bigskip
\noindent{\bf Acknowledgements.}

\noindent{\bf Declaration of interests.} The authors have no conflicts of interest to disclose.

\noindent{\bf Data availability.} No new data were created or analyzed in this study.

\end{document}